\documentclass[a4paper,reqno,oneside]{amsart}
\usepackage{amsmath,amssymb,amsthm,amsfonts,mathrsfs,amsopn} 
\usepackage{enumitem,dsfont}
\usepackage{xcolor}
\usepackage{tikz-cd}
\usepackage{amssymb}

\usepackage[colorlinks=true,urlcolor=blue, citecolor=red,linkcolor=blue,
linktocpage,pdfpagelabels, bookmarksnumbered,bookmarksopen]{hyperref}

\usepackage{xcolor}

\usepackage{geometry}
\newtheorem{thm}{Theorem}
\newtheorem{lem}[thm]{Lemma}
\newtheorem{prop}[thm]{Proposition}
\newtheorem{corol}[thm]{Corollary}
\theoremstyle{definition}
\newtheorem{defin}[thm]{Definition}

\newtheorem{rem}[thm]{Remark}

\newcommand{\C}{\mathbb{C}}
\newcommand{\R}{\mathbb{R}}

\newcommand{\N}{\mathbb{N}}

\newcommand{\cT}{\mathcal{T}}
\newcommand{\cH}{\mathcal{H}}
\newcommand{\cF}{\mathcal{F}}
\newcommand{\cG}{\mathcal{G}}
\newcommand{\cB}{\mathcal{B}}
\newcommand{\cC}{\mathcal{C}}
\newcommand{\cE}{\mathcal{E}}

\newcommand{\cL}{\mathcal{L}}

\newcommand{\cS}{\mathcal{S}}

\newcommand{\cW}{\mathcal{W}}

\newcommand{\cD}{\mathcal{D}}

\newcommand{\eps}{\varepsilon}

\newcommand\dom[1]{\mathrm{dom}\left(#1\right)}

\newcommand\Str{\mathsf{Str}}

\DeclareMathOperator{\dist}{dist}
\DeclareMathOperator{\Sp}{Sp}

\newcommand{\be}{\begin{equation}}
\newcommand{\ee}{\end{equation}}

\newcommand{\Dom}{\mathrm{dom}}
\newcommand{\supp}{\mathop{\mathrm{supp}}\nolimits}
\usepackage[normalem]{ulem}
\definecolor{DarkGreen}{rgb}{0,0.5,0.1} 
\newcommand{\txtD}{\textcolor{DarkGreen}}
\newcommand\soutD{\bgroup\markoverwith
{\textcolor{DarkGreen}{\rule[.5ex]{2pt}{1pt}}}\ULon}
\newcommand{\Hm}[1]{\leavevmode{\marginpar{\tiny%
$\hbox to 0mm{\hspace*{-0.5mm}$\leftarrow$\hss}%
\vcenter{\vrule depth 0.1mm height 0.1mm width \the\marginparwidth}%
\hbox to
0mm{\hss$\rightarrow$\hspace*{-0.5mm}}$\\\relax\raggedright #1}}}

\title[Spectral properties of relativistic quantum waveguides]{Spectral properties of relativistic quantum waveguides}

\author{William Borrelli$^*$}
\address[W. Borrelli]{Dipartimento di Matematica e Fisica, Universit\`a Cattolica del Sacro Cuore, Via Garzetta 48, Brescia, Italy.}
\email{william.borrelli@unicatt.it}
\urladdr{}

\author{Philippe Briet}
\address[P. Briet]{Aix-Marseille Universit\'e, Universit\'e de Toulon, CNRS, CPT, Marseille, France.}
\email{briet@cpt.univ-mrs.fr}
\urladdr{http://www.cpt.univ-mrs.fr/~briet/}

\author{David Krej\v{c}i\v{r}\'ik}
\address[D. Krej\v{c}i\v{r}\'ik]{Department of Mathematics, Faculty of Nuclear Sciences and Physical Engineering, Czech Technical University in Prague, Trojanova 13, 12000 Prage 2}
\email{david.krejcirik@fjfi.cvut.cz}
\urladdr{http://nsa.fjfi.cvut.cz/david/}

\author{Thomas Ourmi\`eres-Bonafos}
\address[T. Ourmi\`eres-Bonafos]{Aix-Marseille Universit\'e, CNRS, Centrale Marseille, I2M, Marseille, France.}

\email{thomas.ourmieres-bonafos@univ-amu.fr}
\urladdr{http://www.i2m.univ-amu.fr/perso/thomas.ourmieres-bonafos/}

\begin{document}

\thanks{$^*$\emph{Corresponding author}. Dipartimento di Matematica e Fisica, Universit\`a Cattolica del Sacro Cuore, Via Garzetta 48, Brescia, Italy. E-mail: william.borrelli@unicatt.it}

\maketitle


\begin{abstract}
We make a spectral analysis of the massive Dirac operator 
in a tubular neighborhood of an unbounded planar curve,  
subject to infinite mass boundary conditions.
Under general assumptions on the curvature,
we locate the essential spectrum and derive an effective Hamiltonian on the base curve
which approximates the original operator in the thin-strip limit. We also investigate the existence of bound states in the non-relativistic limit and give a geometric quantitative condition for the bound states to exist.
\end{abstract}
\medskip

{\footnotesize

\medskip

}


\section{Introduction}

\subsection{Motivations and state of the art}
Consider a massive particle in a guide modelled by 
a uniform tubular neighbourhood of an infinite planar curve. 
A classical particle, 
moving according to Newton's laws of motion
with regular reflections on the boundary,
will eventually leave any bounded set in a finite time,
except for initial conditions of measure zero in the phase space
corresponding to transverse oscillations.
It came as a surprise in 1989 that the situation changes 
drastically for quantum particles modelled by the Schr\"odinger equation.
In the pioneering paper~\cite{ES} 
and further improvements \cite{DE,GJ,KKriz},
it was demonstrated that the quantum Hamiltonian 
identified with the Dirichlet Laplacian
possesses discrete eigenvalues unless the base curve is a straight line. 
Roughly, and with a sharp contrast with the classical setting, 
the particle gets trapped in any non-trivially curved quantum waveguide.
The existence and properties of the geometrically induced 
bound states have attracted a lot of attention in the last decades
and the research field is still very active.
We refer to the monograph~\cite{Exner-Kovarik}
and the latest developments in~\cite{KZ3}
with further references.

The goal of the present paper is to consider relativistic 
counterparts of the quantum waveguides.
Here we model the relativistic quantum Hamiltonian 
by the Dirac operator in the same tubular neighbourhood as above,
subject to infinite mass boundary conditions. 
The latter is probably the reason why the relativistic setting 
has escaped the attention of the community until now.
Indeed, the self-adjointness of the Dirac operators on domains
and the right replacement for the Dirichlet boundary conditions
have been understood only recently 
\cite{Arrizabalaga-LeTreust-Raymond17,
Barbaroux-Cornean-LeTreust-Stockmeyer_2019, 
Benguria-Fournais-Stockmeyer-Bosch_2017b,
LeTreust-Ourmieres-Bonafos_2018}.

There are four motivations for the present study.
First, we would like to understand the influence of 
relativistic effects on spectral properties.
Do the geometrically induced bound states exist 
independently of the mass of the particle?
It is expected that they do exist for heavy particles
because the Dirac operator converges, 
in a suitable sense involving an energy renormalization, 
to the Dirichlet Laplacian in the limit of large masses.
For light particles, however, the answer is far from being obvious
because it is well known that relativistic systems are less stable
\cite{Lieb-Seiringer}.
In this paper we confirm the expectation 
by justifying the non-relativistic limit   
and provide partial 
(both qualitative and quantitative) 
answers for the whole ranges of masses.

Our second motivation is related to quantisation on submanifolds.
It is well known 
(see \cite{KSed} for an overview with many references)
that the non-relativistic quantum Hamiltonian converges 
to a one-dimensional Schr\"odinger operator on the base curve.
(The convergence involving an energy renormalization
can be understood either in a resolvent sense 
\cite{deOliveira_2006,KRRS2,KSed}
or as an adiabatic limit 
\cite{Haag-Lampart-Teufel_2015,Lampart-Teufel_2017,Wachsmuth-Teufel_2013}.)
It is remarkable that this non-relativistic effective operator 
is not the free quantum Hamiltonian on the submanifold
but it contains an extrinsic geometric potential depending 
on the curvature of the base curve. 
In this paper we find that the relativistic setting
is very different, for the limiting operator
describing the effective dynamics on the submanifold    
is just the free Dirac operator of the base curve.

Recently, the Dirac operator on metric graphs has been considered 
as a model for the transport of relativistic quasi-particles in branched structures \cite{YSAEM_20}
and the existence and transport of Dirac solitons in networks have been studied in \cite{SBMK_18}.
Previous studies deal with the quantisation of graphs and spectral statistics for the Dirac operator \cite{BH_03},
and self-adjoint extensions and scattering properties for different graph topologies \cite{BT_90}. 
Rigorous mathematical studies on linear and nonlinear Dirac equations on metric graphs recently appeared \cite{BCT_SIMA17,BCT_SPRINGER20, BCT_19}.
The result of the present paper can be understood as the first step toward a rigorous justification 
of the metric graph model as the limit of shrinking branched waveguides.
 
The last but not least motivation of this paper
is that the present model is relevant for transport
of quasi-particles in graphene nanostructures
\cite{Pedersen-Gunst-Markussen-Pedersen_2012}.
This makes our results not only interesting
in the mathematical context of spectral geometry
and in the physical concept of quantum relativity,
but directly accessible to laboratory experiments
with the modern artificial materials. 
We hope that the present results will stimulate
an experimental verification of the geometrically
induced bound states in graphene waveguides.

\subsection{Geometrical setting and standing hypotheses}
Before presenting our main results in more detail,
let us specify the configuration space of the quantum 
system we are interested in. 

Let $\Gamma \subset \mathbb{R}^2$ be a curve with an injective and $C^3$ arc-length parametrization $\gamma: \mathbb{R} \to \mathbb{R}^2$, i.e., $\gamma(\R)=\Gamma$. 
We define $\nu(s)$ the normal of $\Gamma$ at the point $\gamma(s)$ chosen such that for all $s\in \R$ 
the couple $\big(\gamma'(s),\nu(s)\big)$ is a positive orthonormal basis of~$\R^2$. The curvature of $\Gamma$ at the point $\gamma(s)$, denoted $\kappa(s)$ is defined by the Frenet formula
\begin{equation}\label{Frenet}
	\gamma''(s) = \kappa(s) \nu(s).
\end{equation}	
All along this paper, we make the following assumptions on the curvature $\kappa$:

\begin{enumerate}[label=(\Alph*)]
	\item \label{itm:A}$\displaystyle\lim_{s\to\pm\infty}\kappa(s)=0 $,
	\item \label{itm:B}$\kappa'\in L^\infty(\R)\,.$
\end{enumerate}\noindent
Notice since we work with a $C^3$ curve, $\kappa'$ is automatically continuous, so that assumption $(B)$ implies it is also bounded.
\smallskip

Now, for $ 0< \eps < (\|\kappa\|_{L^\infty(\R)})^{-1}$ (with the convention that the right-hand side
equals $+ \infty$ if $\kappa=0$ identically), 
we define the 
tubular neighbourhood of radius~$\varepsilon$ of~$\Gamma$ in $\R^2$
as the domain
\be\label{eq:strip}
	\Omega_\varepsilon :=\{ \gamma(s) + \varepsilon t \nu(s)
	  : \ s\in \mathbb{R}, \ |t| <  1\}\,,
\ee
that is, $\Omega_\eps$ is the planar strip of width $2\eps$ along the curve $\Gamma$.

It is a well-known result of differential geometry that under these conditions
\be\label{eq:parametrization}
	\Phi_\varepsilon : (s,t) \in \Str \mapsto \gamma(s) + \varepsilon t \nu(s)\in\R^2
\ee
is a local $C^2$-diffeomorphism from the strip 
\begin{equation}\label{eq:Strdef}
\Str := \R\times(-1,1)
\end{equation}
to the set $\Omega_\eps$. In order to ensure that the map $\Phi_\eps$ becomes a global diffeomorphism
we additionally assume that
\begin{enumerate}[label=(\Alph*)]
	\setcounter{enumi}{2}
	\item \label{itm:C} $0< \eps < (2\|\kappa\|_{L^\infty(\R)})^{-1}$ and $\Phi_\eps$ is injective.
\end{enumerate}
Remark that in assumption \ref{itm:C}, one could take $0< \eps < \|\kappa\|_{L^\infty(\R)}^{-1}$ to guarantee that $\Phi_\eps$ is a global diffeomorphism. However, for technical reasons, we need a more restrictive range of admissible width $\eps$.

\begin{rem}
Despite being quite general, assumptions \ref{itm:A}, \ref{itm:B} are probably not optimal. In \cite{KSed} the authors deal with three-dimensional non-relativistic waveguides under minimal technical assumptions on the base curve
(in particular, the curvature does
not need to be differentiable), 
and then similar results can be expected in the present case. However, for ease of presentation we prefer not to investigate this aspect here.
The assumption on the size of $\eps$ in \ref{itm:C} is purely technical and allows to apply Kato's perturbation theory (see the proof of Theorem \ref{prop:sa-essspec}). We mention that another proof could be given adapting the general techniques developed in \cite{Rab} for three-dimensional problems to our setting.
\end{rem}

\subsection{Main results}
We are interested in the relativistic quantum Hamiltonian 
of a (quasi-)particle of (effective) mass $m\geq0$ described by
the Dirac operator with infinite mass boundary conditions posed in the domain $\Omega_\varepsilon$.
Namely, we define the operator $\cD_\Gamma(\varepsilon,m)$ 
in the Hilbert space $L^2(\Omega_\varepsilon,\C^2)$
as
\be\label{eqn:def-dirac}
\begin{split}
	\dom{\cD_\Gamma(\varepsilon,m)} &:= \{ u \in H^1(\Omega_\varepsilon,\C^2) : -i \sigma_3 \sigma\cdot \nu_\varepsilon u = u \text{ on } \partial\Omega_\varepsilon\},\\
	\cD_\Gamma(\varepsilon,m) u  &:= -i \sigma \cdot \nabla u + m \sigma_3 u,
\end{split}
\ee
where $\nu_\varepsilon$ is the outward pointing normal on $\partial\Omega_\varepsilon$. 

In \eqref{eqn:def-dirac} we use the notation $\sigma\cdot v:=\sigma_{1}v+\sigma_{2}v$, where $v\in\C^2$, and $\sigma_{k}$ are the \emph{Pauli matrices}
\begin{equation}\label{eq:pauli} \sigma_{1}:=\begin{pmatrix} 0 \quad& 1 \\ 1 \quad& 0 \end{pmatrix}\quad,\quad \sigma_{2}:=\begin{pmatrix} 0 \quad& -i \\ i \quad& 0 \end{pmatrix} \quad,\quad \sigma_{3}:=\begin{pmatrix} 1 \quad& 0 \\ 0 \quad& -1 \end{pmatrix}\,.\end{equation}
In particular, the action of the operator $\cD_\Gamma$ is given by 
\[
	\cD_\Gamma = \begin{pmatrix}m & -i(\partial_{1}-i\partial_2)\\-i(\partial_{1}+i\partial_2) & -m\end{pmatrix}\,.
\]
As for the boundary conditions, their name is related to the folllowing fact, first recognized in \cite{Berry-Mondragon_1987}. 

Consider the Dirac operator on $L^2(\Omega,\C^2)$, acting as $T:=-i\sigma\cdot\nabla$ and endowed with boundary conditions as in \eqref{eqn:def-dirac}. 
In a sense specified in \cite[Theorem 1.1]{Barbaroux-Cornean-LeTreust-Stockmeyer_2019}, such operator is the (norm-resolvent) limit of Dirac operators on $L^2(\R^2,\C^2),$ of the form $T_M:=-i\sigma\cdot\nabla+\chi_{\R^2\setminus\Omega}M\sigma_3$, as $M\to+\infty$, with a mass term supported outside $\Omega$ (here $\chi$ denotes the characteristic function of a set). This justifies the name \emph{infinite mass} boundary conditions.
\medskip

In what follows we denote by $\Sp$ the spectrum of an operator. Moreover, we shall distinguish between the \emph{discrete spectrum} $\Sp_{\mathrm{dis}}$, namely, the set of eigenvalues of finite multiplicity, and the \emph{essential spectrum} $\Sp_{\mathrm{ess}}=\Sp\setminus\Sp_{\mathrm {dis}}$.
\medskip

Our first result is about the self-adjointness and the structure of the spectrum of $\cD_\Gamma(\varepsilon,m)$.

\begin{thm} \label{prop:sa-essspec}
The operator $\cD_\Gamma(\varepsilon,m)$ defined in \eqref{eqn:def-dirac} is self-adjoint. Its spectrum is symmetric with respect to the origin and there holds:
\[
	\Sp_{\mathrm{ess}}(\cD_\Gamma(\varepsilon,m)) = 
	\big(-\infty ,-\sqrt{\eps^{-2}E_1(m\eps) + m^2}\big] \cup \big[\sqrt{\eps^{-2}E_1(m\eps) + m^2},+\infty\big)\,,
\]
where $E_1(m)$ is the unique
root of the equation
\begin{equation}\label{implicit}
m\sin(2\sqrt{E}) +\sqrt{E}\cos(2\sqrt{E}) = 0
\end{equation}
lying in the line segment $[\frac{\pi^2}{16},\frac{\pi^2}{4})$.
\end{thm}

In order to prove Theorem \ref{prop:sa-essspec}, a first step is to study the operator $\cD_\Gamma(\eps,m)$ in the special case of $\Omega_\eps$ being a straight strip. 
In this setting, a partial Fourier transform gives a fiber decomposition of the operator $\cD_\Gamma(\eps,m)$ and we are left with the investigation of one-dimensional operators which can be understood explicitly.

The second step is to show that in the case of general waveguides $\Omega_\eps$, $\cD_\Gamma(\eps,m)$ can be seen as a perturbation of the operator in the straight strip. To this aim, we will use the following proposition, which allows to work with the $\eps$-independent Hilbert space $L^2(\Str,\C^2)$.

\begin{prop}\label{prop:equiv}
The operator $\cD_\Gamma(\varepsilon,m)$ defined in \eqref{eqn:def-dirac} is unitarily equivalent to the operator $\cE_\Gamma(\varepsilon,m)$ defined on $L^2(\Str,\C^2)$ as:
\[
\begin{aligned}
	\cE_\Gamma(\eps,m) &:= \frac{1}{1-\eps t\kappa}(-i\sigma_1)\partial_s +\frac{1}{\eps}(-i\sigma_2)\partial_t+\frac{\eps t\kappa'}{2(1-\eps t\kappa)^2}(-i\sigma_1)+m\sigma_3\,, \\
	\dom{\cE_\Gamma(\varepsilon,m)}
	&:=\{u =(u_1,u_2)^\top\in H^1(\Str,\C^2)\,:\ u_2(\cdot,\pm1)=\mp u_1(\cdot,\pm1)\,\}.
\end{aligned}
\]
\end{prop}

The main novelty here lies in a matrix-valued gauge transform involving the geometry of the base curve $\Gamma$ in order to deal with the infinite mass boundary conditions. In particular, compared to similar strategies for non-relativistic waveguides, it allows to gauge out one part of the geometric induced potential.

The next two main results of this paper concern the study of the operator $\cD_\Gamma(\varepsilon,m)$ in the thin waveguide asymptotic regime $\eps \to 0$ and the large mass regime $m\to +\infty$, respectively. It turns out that up to renormalization terms, both regimes are driven by effective operators but of very distinct kind. In the thin waveguide regime $\eps \to 0$ the effective operator is a one dimensional Dirac operator posed on the base curve $\Gamma$ while in the large mass regime $m \to +\infty$, the operator behaves as the Dirichlet Laplacian in the domain $\Omega_\eps$. Remark that in both regime Theorem \ref{prop:sa-essspec} combined with the forthcoming Proposition \ref{prop:op1Dess} yields spectral gaps of orders $\eps^{-1}$ and $m$ for the thin waveguide regime and the large mass regime, respectively.

Note that an interesting challenge would be to consider combined regimes in which $\varepsilon\to0$ and $m\to+\infty$ at the same time, and to understand if other effective operators come into play.
Finally, our last result is a quantitative result on the existence of bound states involving only the geometry of the domain $\Omega_\eps$.

\subsubsection{Main result in the thin waveguide regime $\eps \to 0$}\label{par:thinresult}

In this paragraph, we fix $m\geq 0$ and our result in the thin waveguide regime $\varepsilon\to 0$ deals with the existence at first order, up to a renormalization term, of an effective operator. This effective operator is the one dimensional Dirac operator
\begin{equation}\label{eqn:dir1Ddef}
	\cD_{\rm 1D}(m)u := 
	-i\sigma_1\partial_s u + m\sigma_3u,\quad u \in \dom{\cD_{\rm 1D}(m)} := H^1(\R,\C^2).
\end{equation}
It is well known that $\cD_{\rm 1D}(m)$ is a self-adjoint operator with purely absolutely continuous spectrum $\Sp(\cD_{\rm 1D}) = (-\infty,-m]\cup [m,+\infty)$, as can be seen performing a Fourier transform (see \cite[Thm.~1.1]{Thaller} for the analogue in dimension three).

Since this operator acts in $L^2(\R,\C^2)$, it is more convenient to work in the $\varepsilon$-independent Hilbert space $L^2(\Str,\C^2)$ and with the unitarily equivalent operator $\cE_\Gamma(\eps,m)$ introduced in Proposition \ref{prop:equiv}.

\begin{thm}[Thin width limit]\label{thm:thinguide}
There exists a closed subspace $F\subset L^2(\Str,\C^2)$ and a unitary map $V$ such that $V : L^2(\Str,\C^2) \to L^2(\R,\C^2)\oplus F $ and for $\eps\to 0$ there holds
\begin{equation}\label{eq:resolvent_convergence}
	V \, \Big(\cE_{\Gamma}(\varepsilon,m)- \frac{\pi}{4\varepsilon}(P^+ - P^-)-i\Big)^{-1} \, V^{-1} 
	= (\cD_{\rm 1D}(m_e)-i)^{-1} \oplus 0 + \mathcal{O}(\varepsilon)\,,
\end{equation}
in the operator norm, where $P^\pm$ are explicit orthogonal projectors in $L^2(\Str,\C^2)$ and where the effective mass $m_e$ is given by 
$m_e := \frac2\pi m$.
\end{thm}
The projectors $P^\pm$ in the renormalization term of Theorem \ref{thm:thinguide} are projectors on positive and negative spectral subspaces of a one-dimensional transverse Dirac operator. It is remarkable that the geometry of the base curve $\Gamma$ only appears at higher order terms. We do not know if for $\eps$ small enough $\Sp_{\mathrm {dis}}(\cD_{\Gamma}(\varepsilon,m)) \neq \emptyset$. In particular, it would be interesting to investigate further the remainder term in Theorem \ref{thm:thinguide} to understand if the geometry can play a role in the creation of bound states.

Once again, the proof of Theorem \ref{thm:thinguide} is divided in two steps. We first prove Theorem \ref{thm:thinguide} in the special case of $\Omega_\eps$ being a straight strip \textit{via} a projection on the modes of a one-dimensional transverse Dirac operator. The obtained operator can be seen as a block operator $2\times2$ matrix and the main difficulty here lies in the fact that if the mass $m$ is non-zero there are off-diagonal terms. They are handled using Schur's complement theory but a special care is needed in order to control the $\eps$-dependence of each term.

In the second step, we use a perturbation argument to prove that the general waveguides $\Omega_\eps$ can be seen as a perturbation of sufficiently high order of the special case of the straight strip. This step requires a thorough control in $\eps$ of the norm of the resolvent of some operators.

\subsubsection{Large mass regime $m\to+\infty$} In order to state our results in the large mass regime $m\to +\infty$ we need a few notation and definition. First, all along the paper 
$\N := \{1,2,\dots\}$ denotes the set of positive natural integers. We also recall the well-known definitions of the min-max values as well as the min-max principle 
(see \cite[Thm.~4.5.1 \& 4.5.2]{Davies}).

\begin{defin}\label{def:minmax}
Let $\mathfrak{q}$ be a closed lower semi-bounded below quadratic form with dense domain $\Dom(\mathfrak{q})$ in a complex Hilbert space $\cH$. For $n\in\N$, the $n$-th min-max value of $\mathfrak{q}$ is defined as
\be\label{eq:minmaxdef}
\mu_{n}(\mathfrak{q}):=\mathop{\inf_{W\subset\Dom(\mathfrak{q})}}_{\dim W=n}\sup_{u\in W\setminus\{0\}}\frac{\mathfrak{q}(u)}{\| u\|_{\cH}^2}.
\ee
\end{defin}
We also denote by $\mathfrak{q}$ the associated sesquilinear form. If $A$ is the unique self-adjoint operator acting on $\cH$ associated with the sesquilinear form $\mathfrak{q}$ \textit{via} Kato's first representation theorem (see \cite[Ch.~VI, Thm.~2.1]{Kato2})), we shall refer to \eqref{eq:minmaxdef} as the $n$-th min-max value of $A$ and set $\mu_n(A):=\mu_n(\mathfrak{q})$.
\begin{prop}[min-max principle]\label{prop:min-max} Let $\mathfrak{q}$ be a closed semi-bounded below quadratic form with dense domain in a Hilbert space $\cH$ and let $A$ be the unique self-adjoint operator associated  with $\mathfrak{q}$. Then, for $n\in\N$, we have the following alternative:
\begin{enumerate}
	\item if $\mu_n(A) < \inf \Sp_{\mathrm {ess}}(A)$ then $\mu_n(\mathfrak{q})$ is the  $n$-th eigenvalue of $A$ (counted with multiplicity),
	\item if $\mu_n(A) = \inf \Sp_{\mathrm {ess}}(A)$ then for all $k\geq n$ there holds $\mu_k(\mathfrak{q}) = \inf \Sp_{\mathrm {ess}}(A)$.
\end{enumerate}
\end{prop}

Now, we fix $\eps >0$ as we are interested in the large mass regime $m\to +\infty$. Up to an adequate renormalization, this limit can be interpreted as a non-relativistic limit and the Dirichlet Laplacian is expected to be the effective operator in this case (see \cite[Sec.~6]{Thaller} for general remarks on this limit). To this aim, we introduce $\cL_\Gamma(\varepsilon)$, the (spinorial)
Dirichlet Laplacian in the waveguide $\Omega_\varepsilon$, defined by
\be\label{eq:Laplacian}
	\cL_\Gamma(\varepsilon) := - \Delta,\quad 
	\dom{\cL_\Gamma(\varepsilon)} :=
	H_0^1(\Omega_\varepsilon,\C^2)\cap H^2(\Omega_\varepsilon,\C^2).
\ee
Observe that the Sobolev space $H_0^1(\Omega_\varepsilon,\C^2)$ of spinors consists of $\C^2$-valued functions $f=(f_1,f_1)^\top$ such that the components $f_j$ belong to the ordinary (scalar-valued) Sobolev space $f\in H^{1}_0(\Omega)$. The same remarks, of course, applies to $H^2(\Omega_\varepsilon,\C^2)$ and to the other Sobolev spaces of spinors involved in the text.
\medskip

The following proposition summarizes results established
in \cite{DE,KKriz}.
\begin{prop}\label{prop:specdirlapl}
	$\cL_\Gamma(\eps)$ is self-adjoint and there holds
	\[
		\Sp_{\mathrm {ess}}(\cL_\Gamma(\eps)) = \Big[\frac{\pi^2}{4\eps^2},+\infty\Big).
	\]
Moreover, if $\Gamma$ is not a straight line, then there exists $N_\Gamma \in \N \cup \{+\infty\}$ such that
\be\label{eqn:spdislapl2}
	\sharp \Sp_{\mathrm {dis}}(\cL_\Gamma(\eps)) = 2 N_\Gamma.
\ee
\end{prop}
The factor $2$ in \eqref{eqn:spdislapl2} comes from the fact that in \eqref{eq:Laplacian} we consider the Dirichlet Laplacian acting on $\C^2$-valued functions instead of the usual scalar one. In particular, any eigenvalue of $\cL_\Gamma(\eps)$ has even multiplicity. Here, we use the convention that if $N_\Gamma = +\infty$ then $2N_\Gamma = +\infty$.

Our first result in the large mass regime reads as follows.
\begin{prop}\label{cor:evnumber} Let us assume that $\Gamma$ is not a straight line, fix $\varepsilon >0$ and let $n \in \{1,\dots,N_\Gamma\}$. There exists $m_0 >0$ such that for all $m > m_0$
\[
	\# \Sp_{\mathrm dis}(\cD_\Gamma(\varepsilon,m)) \geq 2n.
\]
\label{thm:2}
\end{prop}

Proposition \ref{thm:2} is proved by comparing the quadratic forms of the renormalized operator $\cD_\Gamma(\varepsilon,m)^2-m^2$ to the quadratic form of $\cL_\Gamma(\eps)$, using the min-max principle (Proposition~\ref{prop:min-max}), the asymptotic behavior of $E_1(m)$ when $m\to +\infty$ and Proposition~\ref{prop:specdirlapl}.

Actually, one can show that all the min-max values of the renormalized operator $\cD_\Gamma(\varepsilon,m)^2-m^2$ converge to those of the Dirichlet Laplacian in the regime $m\to+\infty$. This is the purpose of the following theorem. 
 
\begin{thm}[\emph{Large mass limit}]\label{thm:nonrelativistic} 
Let us assume additionally that $\Gamma$ is of class~$C^4$ and that $\kappa'(s) \to 0$ and $\kappa''(s) \to 0$ when $|s|\to + \infty$. Then for all $n\in \N$ there holds:  
\be\label{eq:min-maxconv}
	\lim_{m\to+\infty}\big(\mu_{n}(\cD^2_\Gamma(\varepsilon,m)) - m^2\big)=\mu_{n}(\cL_\Gamma(\varepsilon))\,.
\ee
\end{thm}
In particular,  consider the positive part of the operator $\cD_\Gamma(\varepsilon,m)$ defined by $\cD_\Gamma^+(\varepsilon,m) := \mathds{1}_{x>0}(\cD_\Gamma(\varepsilon,m))$ . Since the spectrum of $\cD_\Gamma(\varepsilon,m)$ is symmetric with respect to zero, under the hypothesis of Theorem \ref{thm:nonrelativistic}, we obtain for all $n\in \N$
\be
	\mu_n(\cD_\Gamma^+(\varepsilon,m)) = m + \frac1{2m}\mu_{2n}(\cL_\Gamma(\eps)) + o\big(\frac1m\big), \quad m \to +\infty;
	\label{eqn:asdevnrl}
\ee
where we have taken into account that the spectrum of $\cL_\Gamma(\eps)$ has even multiplicity. Asymptotics \eqref{eqn:asdevnrl} illustrates the physically expected fact that in the large mass regime $m\to +\infty$, the positive part of the Dirac operator with infinite mass boundary condition converges to the scalar Dirichlet Laplacian. The main novelty in Theorem \ref{thm:nonrelativistic} with respect to the previous work \cite{Arrizabalaga-LeTreust-Raymond17} is that we have to deal with the unbounded domain $\Omega_\eps$. This difficulty is overcome by a standard argument, approximating the min-max values of $\cD^2_\Gamma(\varepsilon,m) - m^2$ by those of similar operators in bounded waveguides using the so-called IMS localization formula (see \cite[Thm.~3.2]{CFKS}).

\subsection{Outline of the paper} 
Section \ref{sec:straight} deals with the infinite mass Dirac operator in the straight strip and with the study of a one-dimensional Dirac operator on a finite interval, 
obtained by separating variables.

Then, in Section \ref{sec:flat}, we show that the Hamiltonian \eqref{eqn:def-dirac} is unitarily equivalent to a Dirac operator in a straight strip, perturbed by a term encoding the geometric properties of the waveguide. Using operator-theoretic methods we are able to prove the self-adjointness and to locate the essential spectrum, as stated in Theorem \ref{prop:sa-essspec}.

Section \ref{sec:thin} is devoted to the proof of Theorem \ref{thm:thinguide}, which is achieved in two steps. First, we deal with the case of the straight waveguide and second, we add the perturbation induced by the curvature. A careful analysis of the resolvent operator allows to prove that, after a suitable renormalization, the Hamiltonian \eqref{eqn:def-dirac} converges in the norm resolvent sense to that of a one dimensional Dirac operator on the line.

Section \ref{sec:nonrel} contains the proof of Theorem \ref{thm:nonrelativistic}, showing that in the large mass regime the min-max values of the (renormalized) squared Hamiltonian converge to those of the vectorial Dirichlet Laplacian $\cL_\Gamma(\eps)$.

Finally, in Section \ref{sec:quantitative}, we obtain a quantitative condition for the existence of at least two bound states in the gap of the essential spectrum. Even though the existence of bound states can be obtained as a corollary of Proposition~\ref{thm:2}, 
we mention this alternative proof because here the condition is given by a simple inequality involving geometric properties of the waveguide $\Omega_\eps$.

\section{Straight waveguides}\label{sec:straight}
In this section we collect results concerning some auxiliary one-dimensional operators that naturally appear in the study of the Hamiltonian \eqref{eqn:def-dirac} in the thin waveguide regime. In order to simplify the overall presentation we postpone their proofs to the Appendix \ref{technical}.
\subsection{The transverse Dirac operator}\label{subsec:transdir}
For $k \in \mathbb{R}$, consider the one-dimensional transverse Dirac operator
\begin{align}
\cT(k,m) &:=  -i  \sigma_2\frac{d}{dt} + k \sigma_1 + m\sigma_3,
\nonumber \\
	\label{eqn:defA}\dom{\cT(k,m)} &:= \{ u=(u_1,u_2)^\top\in H^1\big((-1,1),\C^2\big), u_2(\pm 1) = \mp u_1(\pm1)\}.
\end{align}
The following proposition holds true.
\begin{prop}\label{prop:op1Dess} 
Let $k\in \mathbb{R}$, $m\geq0$. The operator $\cT(k,m)$ is self-adjoint and has compact resolvent. Moreover, the following holds:
\begin{enumerate}[label=(\roman*)]
	\item\label{itm:pt1} $\Sp\big(\cT(k,m)\big)\cap \Big[-\sqrt{m^2 + k^2}, \sqrt{m^2 +k^2}\Big] = \emptyset$,
	\item\label{itm:pt2} the spectrum of $\cT(k,m)$ is symmetric with respect to zero and can be represented as $\Sp\big(\cT(k,m)\big) = \bigcup_{p \geq 1} \{\pm \sqrt{m^2 +k^2 + E_p(m)}\}$, with $E_p(m)>0$ for all $p\geq1$,
	\item\label{itm:pt3} for all $p\in \N$, 
	$E_p(m)$ is the only root lying in $\big[(2p-1)^2\frac{\pi^2}{16},p^2\frac{\pi^2}{4}\big)$ of~\eqref{implicit},
	\item\label{itm:pt4} there holds
		\[
			E_1(m) = \frac{\pi^2}{16} + m + \mathcal{O}(m^2),\quad\text{when } m\to 0,
		\]
	\item\label{itm:pt5}there holds
		\[
			E_1(m) = \frac{\pi^2}{4} - \frac{\pi^2}{4m} + \mathcal{O}(m^{-2}),\quad\text{when } m\to +\infty.
		\]
\end{enumerate}
\end{prop}
The proof of Proposition \ref{prop:op1Dess} will also yield the following corollary concerning the operator $\cT_0 := \cT(0,0)$, which is of crucial importance in the study of the regime $\varepsilon \to 0$.
\begin{corol}\label{cor:cor1} The operator $\cT_0$ is self-adjoint and has compact resolvent. Its spectrum is symmetric with respect to zero and verifies
\[
	\Sp(\cT_0) = \left\{\pm k\frac\pi4: \ k\in\mathbb{N}\right\}.
\]
Corresponding normalized eigenfunctions are given by
\[
	u_k^\pm(t) := \frac12\cos\left(k\frac\pi4(t+1)\right)\begin{pmatrix}1\\1\end{pmatrix} \pm \frac12 \sin\left(k\frac\pi4(t+1)\right)\begin{pmatrix}1\\-1\end{pmatrix}.
\]
\end{corol}


\subsection{The Dirac operator in 
the straight strip}\label{sec:Dirachorizontal}

As will be seen further on in Section \ref{sec:flat}, Theorem \ref{prop:sa-essspec} can be obtained \textit{via} classical perturbation theory arguments. They rely on the fact that the operator $\cD_\Gamma(\varepsilon,m)$ can be seen as a perturbation of the operator $\cD_{\Gamma_0}(\varepsilon,m)$ 
in the straight strip 
$\Str(\varepsilon) := \R\times(-\varepsilon,\varepsilon)$.
Here the base curve $\Gamma_0 := \R \times \{0\}$ is a straight line,
which we parametrize by $\gamma_0(s) := s(1,0)$.
The aim of this paragraph is to prove Theorem \ref{prop:sa-essspec} in this special case.
\begin{prop}\label{prop:specH0} Let $\varepsilon > 0 $. The operator $\cD_{\Gamma_0}(\varepsilon,m)$ is self-adjoint on its domain. Moreover, there holds
\begin{align*}
\Sp\big(\cD_{\Gamma_0}(\varepsilon,m)\big)
&=\Sp_{\mathrm {ess}}\big(\cD_{\Gamma_0}(\varepsilon,m)\big)\\
&=\big(-\infty ,-\sqrt{\eps^{-2}E_1(m\eps) + m^2}\big] 
\cup \big[\sqrt{\eps^{-2}E_1(m\eps) + m^2},+\infty\big)\,,
\end{align*}
where $E_1(m)$ is defined in Theorem \ref{prop:sa-essspec}. 
\end{prop}

In order to work with operators defined on a fixed geometrical domain, we recall that $\Str = \Str(1)$
and consider the unitary map
\[
	U : L^2\big(\Str(\varepsilon),\C^2\big) \to L^2\big(\Str, \C^2\big),\quad (U v)(x) := \sqrt{\varepsilon}v(x_1,\varepsilon x_2).
\]
The operator $\cE_0(\varepsilon,m) := U \cD_{\Gamma_0}(\varepsilon,m)U^{-1}$ verifies
\begin{equation}\label{eqn:defL0}
	\cE_0(\varepsilon,m) = -i \sigma_1 \partial_s - i\varepsilon^{-1} \sigma_2\partial_t + m \sigma_3
\end{equation}
with domain $\dom{\cE_0(\varepsilon,m)} = U \dom{\cD_{\Gamma_0}(\varepsilon,m)}$ which rewrites as
\begin{equation}\label{eqn:domainL0}
	\dom{\cE_0(\varepsilon,m)} = \{u=(u_1,u_2)^\top\in H^1(\Str,\C^2) : u_2(\cdot,\pm1) = \mp u_1(\cdot,\pm1)\}.
\end{equation}
In \eqref{eqn:defL0}, we have used the new coordinates $(s,t)\in\Str$ defined by $s = x_1$ and $t = \eps^{-1} x_2$.

Now, we are in a position to prove Proposition \ref{prop:specH0}. We work with the unitarily equivalent operator $\cE_0(\varepsilon,m)$ rather than the operator $\cD_{\Gamma_0}(\varepsilon,m)$ and the proof relies on a direct integral decomposition of the operator $\cE_0(\varepsilon,m)$ as presented, e.g., in \cite[\S XIII.16.]{RS4}.

\begin{proof}[Proof of Proposition \ref{prop:specH0}]
Consider the unitary partial Fourier transform in the $s$-variable
$$
\cF: L^2(\Str,\C^2)\to L^2(\Str,\C^2)\,,\qquad 
(\cF u)(k,t) :=
\frac{1}{\sqrt{2\pi}}\int_\R e^{-i s k}u(s,t)d s\,.
$$
The operator $\cE_0(\varepsilon,m)$ is unitarily equivalent to the direct integral
$$
\cE_0(\varepsilon,m)= \cF^{-1}\widehat{\cE_0}(\varepsilon,m)\cF,\quad \widehat{\cE_0}(\varepsilon,m) := \int^{\oplus}_\R\widehat{\cE_0}(\varepsilon,m;k)d k ,
$$
where $\dom{\widehat{\cE_0}(\varepsilon,m)}$ is the subspace of functions $u = (u_1,u_2)^\top \in L^2(\Str,\C^2)$ such that for almost all $k \in \R $ we have $\partial_{t} u(k,\cdot) \in L^2((-1,1),\C^2)$, $u_2(k,\pm 1) = \mp u_1(k,\pm 1)$ and for almost all $t \in (-1,1)$ there holds $\int_\R k^2|u(k,t)|^2 dk < +\infty$.

One observes that $\widehat{\cE_0}(\varepsilon,m;k)$ satisfies $ \widehat{\cE_0}(\varepsilon,m;k) = \frac{1}\eps\cT(k\varepsilon,m\varepsilon)$ where the operator $\cT(\cdot,\cdot)$ is defined in \eqref{eqn:defA}. In particular, $\widehat{\cE_0}(\varepsilon,m;k)$ is self-adjoint and so is $\widehat{\cE_0}(\varepsilon,m)$ by \cite[Thm. XIII.85 (a)]{RS4}. In particular, we have proved that $\cE_0(\varepsilon,m)$ is a self-adjoint operator.

By \cite[Thm.~XIII.85 (d)]{RS4} there holds
\be\label{eq:kspectrum}
\Sp(\cE_0(\varepsilon,m))=\bigcup_{k\in\R}\Sp(\widehat{\cE_0}(\eps,m;k))\,.
\ee
Remark that we have
\[
\Sp(\widehat{\cE_0}(\eps,m;k)) = \eps^{-1}\Sp(\cT(k\eps,m\eps)) = \bigcup_{p\in\mathbb{N}} 
\left\{\pm \sqrt{m^2 + k^2 +\eps^{-2}E_p(\eps m)}\right\}.
\]
By \ref{itm:pt3}~Proposition \ref{prop:op1Dess}, for all $p\geq 1$, $E_p(\eps m) \in \big[(2p-1)^2 \frac{\pi^2}{16},p^2\frac{\pi^2}{4}\big)$. In particular, there holds
\[
	\Sp \big(\cE_0(\eps,m)\big) = \big(-\infty, - \sqrt{m^2 + \eps^{-2}E_1(\eps m)}\big] \cup \big[\sqrt{m^2 + \eps^{-2}E_1(\eps m)},+\infty\big).
\]
It concludes the proof of Proposition \ref{prop:specH0}.
\end{proof}

\section{First properties in curved waveguides}\label{sec:flat}
The main goal of this section is to prove Theorem \ref{prop:sa-essspec}. As mentioned before, the overall strategy consists in regarding the operator $\cD_\Gamma(\varepsilon,m)$ 
in the curved strip $\Omega_\eps$
as a perturbation of the operator $\cD_{\Gamma_0}(\varepsilon,m)$
in the straight strip $\Str(\eps)$.

In the first paragraph of this section we derive an operator in a straight waveguide, unitarily equivalent to $\cD_\Gamma(\varepsilon,m)$, which is given by a Dirac-type operator in the horizontal strip $\Str = \R\times(-1,1)$ perturbed by a curvature-induced potential. The second and third paragraphs deal with the self-adjointness and the invariance of the essential spectrum, respectively. The key arguments rely on perturbation theory.


\subsection{Straightening the waveguide}
This paragraph is devoted to the proof of Proposition \ref{prop:equiv}. The overall scheme is well-known
in the study of non-relativistic waveguides and numerous works have taken advantage of such a reduction (see, e.g., \cite{DE}). However, we give a complete proof here because the algebraic structure of the Dirac operator allows to gauge out one part of the curvature-induced potential, which appears to be a new effect.

\begin{proof}[Proof of Proposition \ref{prop:equiv}] The proof is divided into three steps. In the first one, we rewrite the problem in tubular coordinates in order to work in the strip $\Str$. The resulting operator acts in a weighted $L^2$-space and we perform a unitary transform in order to work in a non-weighted $L^2$-space; this is the purpose of the second step. Finally, we build a unitary map in order to recover the same boundary condition as the one of the operator $\cE_0(\varepsilon,m)$ investigated in Section~\ref{sec:Dirachorizontal}. This last step partially simplifies the curvature-induced potential.

\textbf{Step 1.}
Consider the unitary map
\be\label{eq:U_1}
U_1:L^2(\Omega_\eps,\C^2)\longrightarrow L^{2}(\Str,\C^2; gdsdt),\quad (U_1u)(s,t) := u(\Phi_\eps(s,t)),
\ee
where  $\Phi_\eps$ is the parametrization of the waveguide given in \eqref{eq:parametrization} and where 
$$
  g(s,t) := \eps\big(1-\eps t\kappa(s)\big)
  \,.
$$
Next, we consider the operator $\cD_{\Gamma,1}(\varepsilon,m) := U_1\cD_\Gamma(\varepsilon,m)U_1^{-1}$. One sees that its domain is 
\begin{align*}
\dom{\cD_{\Gamma,1}(\varepsilon,m)} 
&= U_1 \dom{\cD_\Gamma(\varepsilon,m)} \\
&= \Big	\{ u=(u_1,u_2)^\top \in L^2(\Str,\C^2;gdsdt) : \\
& \qquad (1-\eps t \kappa)^{-1}\partial_su,\partial_tu \in L^2(\Str,\C^2;gdsdt),\\
&\qquad \text{for all } s\in \R\ u_2(s,\pm 1) = \pm i {\bf n}(s)u_1(s,\pm 1) \Big\},
\end{align*}
where for $s\in\R$ we have set 
${\bf n}(s) := \nu_1(s) + i\nu_2(s)$. The operator $\cD_{\Gamma,1}(\varepsilon,m)$ acts on $u \in \dom{\cD_{\Gamma,1}(\varepsilon,m)}$ as
\[
	\cD_{\Gamma,1}(\varepsilon,m) u = - \frac{i}{1-\eps t\kappa} \sigma_{\gamma'} \partial_s u - \frac{i}{\eps} \sigma_{\nu} \partial_t u + m \sigma_3 u,
\]
where for $x=(x_1,x_2) \in \R^2$ we have set $\sigma_x := \sigma\cdot x$.

\textbf{Step 2.} In order to flatten the metric, consider the unitary map
\begin{equation}\label{eq:U_2}
U_2:L^{2}(\Str,\C^2; gdsdt)\longrightarrow L^{2}(\Str,\C^2),\quad U_2u := \sqrt{g}u\,.
\end{equation}
Let $\cD_{\Gamma,2}(\varepsilon,m) := U_2 U_1 \cD_\Gamma(\eps,m)U_1^{-1} U_2^{-1} = U_2 \cD_{\Gamma,1}(\varepsilon,m)U_2^{-1}$. The domain of $\cD_{\Gamma,2}(\eps,m)$ is given by
\begin{align*}
\dom{\cD_{\Gamma,2}(\eps,m)} & = U_2 \dom{\cD_{\Gamma,1}(\eps,m)} \\
&= \big\{ u=(u_1,u_2)^\top \in H^1(\Str,\C^2) :\\&\qquad \text{for all } s\in \R\ u_2(s,\pm 1) = \pm i {\bf n}(s)u_1(s,\pm 1)\big\},
\end{align*}
and for $u \in \dom{\cD_{\Gamma,2}(\eps,m)}$ the operator $\cD_{\Gamma,2}(\eps,m)$ acts as
\begin{align*}
	\cD_{\Gamma,2}(\eps,m) u = &\quad\frac1{1-\eps t\kappa}(-i \sigma_{\gamma'})\partial_s u + \frac1\eps(-i \sigma_{\nu})\partial_t u\\
	&\qquad  + \frac{\eps t\kappa'}{2(1-\eps t \kappa)^2}(-i \sigma_{\gamma'})u + \frac{\kappa}{2(1-\eps t \kappa)}(-i
	 \sigma_{\nu})u +m\sigma_3 u.
\end{align*}
Note that the $H^1(\Str,\C^2)$ regularity of functions in $\dom{\cD_{\Gamma,2}(\eps,m)}$ is a consequence of the regularity hypothesis on the curve $\Gamma$ (see \ref{itm:A} and \ref{itm:B}).

\textbf{Step 3.} Recall that ${\bf n}= \nu_1 + i \nu_2$ and by the Frenet formula~\eqref{Frenet}
we have $ {\bf n}' = i \kappa {\bf n}$. In particular, there holds
\[
	{\bf n}(s) = 
	\exp\left(i \int_{0}^s\kappa(\xi)d\xi\right)
	\bf{n}_0.
\]
where we have set ${\bf n}_0:= {\bf n}(0)$. Moreover, there exists $\theta_0 \in \R$ such that ${\bf n}_0 := e^{i \theta_0}$. 
By setting
\be
\theta(s) := \theta_0 + \int_{0}^s\kappa(\xi)d\xi \,,
\label{eqn:deftheta}
\ee
we get ${\bf n}(s) = \exp(i \theta(s))$. 
For any fixed $s\in \R$, consider the unitary matrix
\[
	U_\theta(s) := 	\begin{pmatrix}
					\exp\big(i(\frac\pi4 + \frac{1}2 \theta(s))\big) & 0\\
					0 & -\exp\big(-i(\frac\pi4 +\frac{1}2 \theta(s)\big)
				\end{pmatrix} .
\]
Note that the mapping
$s\in\R \mapsto U_\theta(s) \in \C^{2\times2}$ is of class $C^2(\R)$.
In order to obtain a boundary condition independent of the normal vector $\nu$ we introduce the unitary map
\begin{equation}\label{eq:U_3}
	U_3 : L^2(\Str,\C^2) \longrightarrow L^2(\Str,\C^2),\quad 
	U_3 u := U_\theta u.
\end{equation}
The operator $\cD_{\Gamma,3}(\eps,m) := U_3 \cD_{\Gamma,2}(\eps,m) U_3^{-1}$ is unitarily equivalent to $\cD_\Gamma(\eps,m)$. As $U_\theta$ is a bounded and $C^2(\R)$ function, its domain is given by
\begin{align*}
\dom{\cD_{\Gamma,3}(\eps,m)} 
&= U_3 \dom{\cD_{\Gamma,2}(\eps,m)} \\
&= \big\{ u=(u_1,u_2)^\top \in H^1(\Str,\C^2) :\\&\qquad \text{for all } s\in \R\ u_2(s,\pm 1) = \mp  u_1(s,\pm 1)\big\}.
\end{align*}
Remark that other choices are possible for the matrices $U_\theta$ but the present one gives the same boundary condition as in the straight waveguide case. Moreover, for $u \in \dom{\cD_{\Gamma,3}(\eps,m)}$, there holds
\begin{align*}
	\cD_{\Gamma,3}(\eps,m) u = &\quad\frac1{1-\eps t\kappa}U_\theta(-i \sigma_{\gamma'})\partial_s (U_\theta^*u) + \frac1\eps(-i \sigma_2)\partial_t u\\
	&\qquad  + \frac{\eps t\kappa'}{2(1-\eps t \kappa)^2}(-i \sigma_1)u + \frac{\kappa}{2(1-\eps t \kappa)}(-i \sigma_2)u +m\sigma_3 u,
\end{align*}
where we have used the identities
\[
	U_\theta \sigma_{\gamma'}U_\theta^* = \sigma_1, \quad U_\theta \sigma_{\nu}U_\theta^* = \sigma_2,\quad U_\theta \sigma_{3}U_\theta^* = \sigma_3.
\]
One also obtains
\[
	U_\theta (-i \sigma_{\gamma'})\partial_s(U_\theta^*) = \frac{\kappa}2(i \sigma_2)
\]
which finally gives
\begin{align*}
	\cD_{\Gamma,3}(\eps,m) u = &\quad\frac1{1-\eps t\kappa}(-i \sigma_1)\partial_s u + \frac1\eps(-i \sigma_2)\partial_t u\\
	&\qquad  + \frac{\eps t\kappa'}{2(1-\eps t \kappa)^2}(-i \sigma_1)u +m\sigma_3 u.
\end{align*}
The proof is completed by setting 
$\cE_\Gamma(\eps,m) := \cD_{\Gamma,3}(\eps,m)$.
\end{proof}

\subsection{Quadratic form of the square}
This section contains an explicit expression of the quadratic form of the square of the operator $\cE_\Gamma(\eps,m)$ defined in Proposition \ref{prop:equiv}. Throughout this section, 
we assume that $\Gamma$ is of class~$C^4$,
in order to give a meaning to $\kappa''$.

\begin{prop}\label{prop:quadform2} Let us assume additionally that $\Gamma$ is of class~$C^4$. Then, for every
$u \in \dom{\cE_\Gamma(\eps,m)}$, there holds
\[\begin{split}
\|\cE_\Gamma(\eps,m) u\|_{L^2(\Str,\C^2)}^2= \ &\int_{\Str}\frac1{(1-\varepsilon t \kappa)^2}\vert\partial_s u - i\frac\kappa2\sigma_3 u\vert^2 ds dt + \frac1{\varepsilon^2}\int_\Str \vert\partial_t u\vert^2 dsdt\\
&+\frac{m}{\eps}\int_{\R}\left(\vert u(s,1)\vert^2+ \vert u(s,-1)\vert^2 \right)ds+m^2\|u\|_{L^2(\Str,\C^2)}^2 \\
&- \int_{\Str}\frac{\kappa^2}{4(1-\varepsilon t \kappa)^2}\vert u\vert^2 ds dt-\frac{5}{4}\int_{\Str}\frac{(\eps t\kappa')^2}{(1-\eps t\kappa)^4}\vert u\vert^2dsdt\\&-\frac{1}{2}\int_{\Str}\frac{\eps t\kappa''}{(1-\eps t\kappa)^3} \vert u\vert^2dsdt\,.
\end{split}
\]
\end{prop}

The proof of Proposition \ref{prop:quadform2} is omitted. It relies on the the next lemma, whose proof follows arguing as in \cite[Lemma 2.1]{Benguria-Fournais-Stockmeyer-Bosch_2017b}. Then the quantity $\|\cE_\Gamma(\eps,m) u\|^2$ for $u\in\dom{\cE_\Gamma(\eps,m)}$ can be simplified performing rather straightforward (but demanding) integration by parts.
\begin{lem} The set $C_0^\infty(\overline{\Str},\C^2)\cap \dom{\cE_0(\eps,0)}$ is dense in $\dom{\cE_0(\eps,0)}$ for the graph norm.
\label{lem:dens}
\end{lem}

\subsection{Self-adjointness}
In this paragraph we prove that $\cD_\Gamma(\varepsilon,m)$ is self-adjoint using the Kato-Rellich theorem 
(see, e.g., \cite[Thm.~4.3.]{Kato2}).
\begin{prop}The operator $\cD_\Gamma(\varepsilon,m)$ is self-adjoint.
\label{prop:kato-rellich}
\end{prop}
Before going through the proof of Proposition \ref{prop:kato-rellich}, we need a few lemmata regarding the operator $\cE_0(\eps,m)$ introduced in \eqref{eqn:defL0}. The first lemma is a consequence of Proposition \ref{prop:quadform2}, taking into account that in this special case $\kappa = 0$ and $m = 0$.

\begin{lem}\label{lem:equiv-norm} For all $u \in \dom{\cE_0(\eps,0)}$, there holds
\[
	\|\cE_0(\eps,0) u\|_{L^2(\Str,\C^2)}^2 = \|\partial_s u\|_{L^2(\Str,\C^2)}^2 + \frac1{\eps^2}\|\partial_t u\|_{L^2(\Str,\C^2)}^2.
\]
\end{lem}

The following Lemma is well-known and follows integrating by parts taking into account the boundary condition.

\begin{lem} The operator $\cD_\Gamma(\eps,m)$ is symmetric.
\label{lem:sym}
\end{lem}

We are now ready to prove Proposition \ref{prop:kato-rellich}.
\begin{proof}[Proof of Proposition \ref{prop:kato-rellich}] Instead of working with the operator $\cD_\Gamma(\varepsilon,m)$, we work with the unitarily equivalent operator $\cE_\Gamma(\varepsilon,m)$ introduced in Proposition \ref{prop:equiv}.  Moreover, as the multiplication operator by $\sigma_3$ is bounded and self-adjoint in $L^2(\Str,\C^2)$ 
we set
$m = 0$ without loss of generality.

Remark that $\dom{\cE_\Gamma(\varepsilon,m)} = \dom{\cE_0(\varepsilon,m)}$ where $\cE_0(\varepsilon,m)$ is defined in \eqref{eqn:defL0} and that for $u\in\dom{\cE_\Gamma(\varepsilon,0)}$ there holds
\[\label{eqn:pert}
	\cE_\Gamma(\varepsilon,0) u =  \cE_0(\varepsilon,0)u + V(\eps),
\]
where the perturbation operator $V(\eps)$ is defined as
\begin{equation}\label{eqn:defpertur}
	V(\eps) := 
	\frac{\eps t\kappa}{1-\eps t\kappa}(-i \sigma_1)\partial_s + \frac{\eps t\kappa'}{2(1-\eps t\kappa)^2}(-i \sigma_1),\quad \dom{V(\eps)} := \dom{\cE_0(\eps,0)}.
\end{equation}
Remark that $V(\eps)$ is a symmetric operator because $V(\eps)$ is the difference of two symmetric operators: $\cE_0(\varepsilon,0)$ is self-adjoint thus symmetric (see Proposition \ref{prop:specH0}) and $\cE_\Gamma(\varepsilon,0)$ is symmetric because it is unitarily equivalent to a symmetric operator (see Lemma \ref{lem:sym} and Proposition \ref{prop:equiv}).

Now, remark that for $u \in C_0^\infty(\overline{\Str},\C^2)\cap\dom{\cE_0(\eps,0)}$, there holds
\[
	\|V(\eps) u\|_{L^2(\Str,\C^2)} \leq \frac{\eps\|\kappa\|_{L^\infty(\R)}}{1-\eps \|\kappa\|_{L^\infty(\R)}}  \|\partial_s u\|_{L^2(\Str,\C^2)} + \frac{\eps \|\kappa'\|_{L^\infty(\R)}}{2\big(1-\eps\|\kappa\|_{L^\infty(\R)}\big)^2} \|u\|_{L^2(\Str,\C^2)}.
\]
Using Lemma \ref{lem:equiv-norm} we obtain
\begin{equation}\label{eqn:contr-pertur}
	\|V(\eps) u\|_{L^2(\Str,\C^2)} \leq \frac{\eps\|\kappa\|_{L^\infty(\R)}}{1-\eps \|\kappa\|_{L^\infty(\R)}}  \|\cE_0(\eps,0) u\|_{L^2(\Str,\C^2)} + \frac{\eps \|\kappa'\|_{L^\infty(\R)}}{2\big(1-\eps\|\kappa\|_{L^\infty(\R)}\big)^2} \|u\|_{L^2(\Str,\C^2)}
\end{equation}
and by density of $C_0^\infty(\overline{\Str},\C^2)\cap\dom{\cE_0(\eps,0)}$ in $\dom{\cE_0(\eps,0)}$ for the graph norm (see Lemma \ref{lem:dens}), \eqref{eqn:contr-pertur} also holds for $u \in \dom{\cE_0(\eps,0)}$.

Remember that we assumed \ref{itm:C}. Hence, we have
\[
	\frac{\eps\|\kappa\|_{L^\infty(\R)}}{1-\eps \|\kappa\|_{L^\infty(\R)}} < 1.
\]
As $V(\eps)$ is symmetric and $\cE_0(\eps,0)$-bounded with $\cE_0(\eps,0)$-bound smaller than $1$ we can apply \cite[Thm.~4.3.]{Kato2} and $\cE_\Gamma(\eps,0)$ is self-adjoint.
\end{proof}

\subsection{Invariance of the essential spectrum}\label{subsec:ess_spec}
In this paragraph we prove that the essential spectrum of $\cE_\Gamma(\eps,m)$ is the same as the one of $\cE_0(\eps,m)$. This is the purpose of the following proposition.
\begin{prop}\label{prop:ess_spec} There holds
\[
	\Sp_{\mathrm {ess}}(\cD_\Gamma(\eps,m)) = \big(-\infty,-\sqrt{m^2 +\eps^{-2}E_1(m\eps)}\big] \cup \big[\sqrt{m^2 +\eps^{-2}E_1(m\eps)},+\infty\big).
\]
\end{prop}

\begin{proof}[Proof of Proposition \ref{prop:ess_spec}] Instead of working with the operator $\cD_\Gamma(\eps,m)$ we work with the unitarily equivalent operator $\cE_\Gamma(\eps,m)$. Our aim is to apply Weyl's criterion \cite[Thm.~XIII.14]{RS4} and for this purpose we define
\[
	\cW := (\cE_\Gamma(\eps,m) +i)^{-1} - (\cE_0(\eps,m) +i)^{-1} 
\]
and by the second resolvent identity one gets $\cW = (\cE_0-i)^{-1} V(\eps)(\cE_\Gamma+i)^{-1}$, where the perturbation $V(\eps)$ is defined in \eqref{eqn:defpertur}.
%

%
Observe that
\[
  V(\eps) = a \partial_s + \partial_s a 
  \,, \qquad \mbox{where} \qquad
  a := \frac{1}{2} \left(\frac{1}{1-\eps\kappa t}-1\right)
  (-i\sigma_1)
  \,.
\]
Then we get
\[
\begin{aligned}
  -\mathcal{W} 
  &=  (\cE_0 - i)^{-1} V(\eps) (\cE_\Gamma + i)^{-1}
  \\
  &= (\cE_0 - i)^{-1} a \ \partial_s (\cE_\Gamma + i)^{-1}
  + (\cE_0 - i)^{-1} \partial_s \ a (\cE_\Gamma + i)^{-1}
  \,.
\end{aligned}
\]
Here $a (\cE_\Gamma + i)^{-1}$
and $(\cE_0 - i)^{-1}a$ are compact operators 
in $L^2(\Str,\C^2)$
due to hypothesis~\ref{itm:A}
(the latter operator is compact because its adjoint 
$a(\cE_0 + i)^{-1}$ is compact).
At the same time, $\partial_s (\cE_\Gamma + i)^{-1}$
and $(\cE_0 - i)^{-1} \partial_s$ are bounded operators 
in $L^2(\Str,\C^2)$ 
(the latter operator is bounded because its adjoint 
$-\partial_s(\cE_0 + i)^{-1}$ is bounded).
Then the compactness of~$\mathcal{W}$ 
follows by the well-known fact that
compact operators are *-both-sided ideal 
in the space of bounded operators.
\end{proof}
\subsection{Proof of Theorem \ref{prop:sa-essspec}}

We are now in a good position to prove Theorem \ref{prop:sa-essspec}.
\begin{proof}[Proof of Theorem \ref{prop:sa-essspec}] Thanks to Proposition \ref{prop:kato-rellich} and Proposition \ref{prop:ess_spec}, the only thing left to prove
is the symmetry of the spectrum of $\cD_\Gamma(\varepsilon,m)$. It is a consequence of the invariance of the system under charge conjugation, corresponding to the operator
\[
	\mathfrak{C} := \sigma_1 C
\]
where $C$ is the complex conjugation operator. A straightforward computation shows that for all $u \in \dom{\cD_\Gamma(\varepsilon,m)}$ we have $\mathfrak{C}u\in\dom{\cD_\Gamma(\varepsilon,m)}$ and
\[
	\cD_\Gamma(\varepsilon,m) (\mathfrak{C} u ) = -\mathfrak{C} \cD_\Gamma(\varepsilon,m) u.
\]
In particular, any Weyl sequence $(u_n)_{n\in\N}$ associated with $\lambda \in Sp(\cD_\Gamma(\varepsilon,m))$ corresponds to a Weyl sequence $(\mathfrak{C}u_n)_{n\in\N}$ associated with $-\lambda$ which proves that the spectrum of $\cD_\Gamma(\eps,m)$ is symmetric and concludes the proof of Theorem \ref{prop:sa-essspec}.
\end{proof}

\section{Thin waveguide limit}\label{sec:thin}
In this section we prove Theorem \ref{thm:thinguide}, which deals with the thin waveguide limit $\eps\to0$. We first show that, up to a renormalization, the operator $\cE_\Gamma(\eps,m)$ 
defined in Proposition \ref{prop:equiv} converges to the one-dimensional Dirac operator \eqref{eqn:dir1Ddef} in the norm resolvent sense. 

The proof is achieved in two different steps. First, in Section \ref{subsec:cvgstr}, we deal with the case of a straight strip and then, in Section \ref{subsec:proofthm3}, we consider the curved waveguide. 

Roughly speaking, the main idea of the proof is to project onto the eigenfunctions of the transverse part of the operator. It turns out that after renormalization, all tranverse modes converge to zero except the first positive and negative one. The operator $\cE_\Gamma(\eps,m)$ restricted to these two modes is unitarily equivalent to a one-dimensional Dirac operator as defined in \eqref{eqn:dir1Ddef}.

\subsection{Convergence for the straight strip}\label{subsec:cvgstr}

For $k\geq 1$, let $\pi_k$ denote the projector in $L^2((-1,1),\C^2)$ on the vector space $\textrm{span}(u_k^+,u_k^-)$, where $u_k^\pm$ are given in Corollary \ref{cor:cor1}. Similarly, we consider the projectors in $L^2((-1,1),\C^2)$ defined by $p^\pm := \mathds{1}_{\{\pm x>0\}}(\cT_0)$ where $\cT_0$ is defined 
in Section~\ref{subsec:transdir}. These projectors can be extended to $L^2(\Str,\C^2)$ setting for $u\in L^2(\Str,\C^2)$
\begin{equation}\label{eq:a_projectors}
\Pi_k u := \pi_k u,\qquad P^\pm u := p^\pm u.
\end{equation}
For further use, we renormalize the operator $\cE_0(\eps,m)$ as follows
\begin{equation}
	\cC(\eps,m) := \cE_0(\eps,m) - \frac\pi{4\eps}\big(P^+ - P^-\big).
	\label{eqn:defopC}
\end{equation}
To investigate the behavior of the resolvent operator $(\cC(\eps,m)-i)^{-1}$ in the thin waveguide regime $\eps \to 0$, we consider the unitary map
\begin{equation}
	U : L^2(\Str,\C^2) \to \Pi_1L^2(\Str,\C^2) \times \Pi_1^\perp L^2(\Str,\C^2), \quad (Uv) 
	:= (\Pi_1 v,\Pi_1^\perp v)^\top,
	\label{eqn:unitthin}
\end{equation}
and remark that there holds
\begin{equation}\label{eq:M}
	U(\cC(\varepsilon,m)-i)U^{-1} = \begin{pmatrix}\cC_1(\varepsilon,m)-i & \Pi_1\cC(\varepsilon,m)\Pi_1^\perp\\\Pi_1^\perp \cC(\varepsilon,m)\Pi_1 & \cC^\perp_1(\varepsilon,m) -i\end{pmatrix},
\end{equation}
where we have set 
\begin{equation}
\cC_1^\perp(\varepsilon,m) := \Pi_1^\perp\cC(\varepsilon,m)\Pi_1^\perp.
\label{eqn:C1}
\end{equation}
For further uses, for all $k\geq 1$ we introduce the operators
\begin{equation}
	\cC_k(\eps,m) := \Pi_k \cC(\eps,m)\Pi_k.
	\label{eqn:Ck}
\end{equation}

In the remaining part of this paragraph, we will make an extensive use of block operator matrix theory to investigate \eqref{eq:M} (see \cite{Tretter} for an extensive discussion).

\subsubsection{A few lemmata}
The first lemma is about the operators $\cC_k(\eps,m)$ defined in \eqref{eqn:Ck}. It states that they are unitarily equivalent 
to one\txtD{-}dimensional Dirac operators (see \eqref{eqn:dir1Ddef}).

\begin{lem}\label{lem:uniequiv1D} Let $k \geq 1$ and consider the unitary map $U_k : \Pi_k L^2(\Str,\C^2) \to L^2(\R,\C^2)$ defined by $U_k v := \begin{pmatrix}\langle v,u_k^+\rangle_{L^2((-1,1),\C^2)}\\\langle v,u_k^-\rangle_{L^2((-1,1),\C^2)}\end{pmatrix}$. There holds
\[
	U_k\cC_k(\varepsilon,m)U_k^{-1} = \cD_{1D}((k-1)\frac{\pi}{4\varepsilon} + m_{e,k})
\]
where $m_{e,k} := \left\{\begin{array}{ll}0 & \text{if } k \text{ is even,}\\
	 \frac{2}{k\pi}m & \text{if } k \text{ is odd.}\end{array}\right.$
In particular, there holds
\[
	\Sp(\cC_k(\varepsilon,m)) = 
	\big(-\infty,-(k-1)\frac{\pi}{4\varepsilon} - m_{e,k}\big] 
	\cup \big[(k-1)\frac{\pi}{4\varepsilon} + m_{e,k},+\infty\big).
\]
\end{lem}
\begin{proof}[Proof of Lemma \ref{lem:uniequiv1D}] Let us pick $f = \begin{pmatrix}f^+\\f^-\end{pmatrix} \in H^{1}(\R,\C^2)$ and consider
\begin{align*}
	\cC_k(\varepsilon,m) U_k^{-1} f =&\ \cC_k(\varepsilon,m) (f^+ u_k^+ + f^- u_k^-) \\
	= & \ \Pi_k ((-i\sigma_1)\partial_s + \frac1\varepsilon(-i\sigma_2)\partial_t + m\sigma_3)(f^+ u_k^+ + f^- u_k^-)\\
	= & \ \big(-i(f^+)'\big) u_k^{-} + \big(-i(f^-)'\big)u_k^+ + (k-1)\frac{\pi}{4\varepsilon}\Big(f^+ u_k^+ - f^- u_k^-\Big)\\
	&\qquad + m\Pi_k (f^+ \sigma_3 u_k^+ + f^- \sigma_3 u_k^-)\\
	 = & \ \big(-i(f^+)'\big) u_k^{-} + \big(-i(f^-)'\big)u_k^+ + (k-1)\frac{\pi}{4\varepsilon}\Big(f^+ u_k^+ - f^- u_k^-\Big)\\
	&\qquad + mf^+ (\langle\sigma_3 u_k^+,u_k^+\rangle_{L^2((-1,1),\C^2)} u_k^+ + \langle\sigma_3 u_k^+,u_k^-\rangle_{L^2((-1,1),\C^2)}u_k^-)\\&\qquad\qquad + mf^- (\langle\sigma_3 u_k^-,u_k^+\rangle_{L^2((-1,1),\C^2)}u_k^+ + \langle\sigma_3 u_k^-,u_k^-\rangle_{L^2((-1,1),\C^2)}u_k^-).\\
\end{align*}
However, using that $\sigma_1 u_k^\pm = u_k^\mp$ as well as the anti-commutation rules of the Pauli matrices we get
\begin{multline*}
	\langle\sigma_3 u_k^+,u_k^+\rangle_{L^2((-1,1),\C^2)} = - \langle \sigma_3 u_k^-,u_k^-\rangle_{L^2((-1,1),\C^2)},\\ \langle\sigma_3 u_k^+,u_k^-\rangle_{L^2((-1,1),\C^2)} =- \langle \sigma_3 u_k^-,u_k^+\rangle_{L^2((-1,1),\C^2)}.
\end{multline*}
Now, a simple computation gives
\begin{equation}
	\langle\sigma_3 u_k^+(t), u_k^-(t)\rangle_{\C^2} = 0,\quad \langle\sigma_3 u_k^+,u_k^+\rangle_{L^2((-1,1),\C^2)} = \left\{\begin{array}{ll}0 & \text{if } k \text{ is even,}\\
	 \frac{2}{k\pi} & \text{if } k \text{ is odd,}\end{array}\right.
	 \label{eqn:effmassobtention}
\end{equation}
and we set $m_{e,k} := m\langle\sigma_3 u_k^+,u_k^+\rangle_{L^2((-1,1),\C^2)}$. In particular, there holds
\begin{align*}
	\cC_k(\varepsilon,m) U_k^{-1} f  = & \ \big(-i(f^+)'\big) u_k^{-} + \big(-i(f^-)'\big)u_k^+ + (k-1)\frac{\pi}{4\varepsilon}\Big(f^+ u_k^+ - f^- u_k^-\Big)\\
	&\qquad + m_{e,k} f^+  u_k^+ -m_{e,k} f^- u_k^+\,,
\end{align*}
so that
\begin{equation}\label{eq:k-mode_Dirac}
	U_k\cC_k(\varepsilon,m) U_k^{-1} f = \Big(-i\sigma_1 \frac{d}{ds} + (m_{e,k} + (k-1)\frac\pi{4\varepsilon})\sigma_3\Big) f = \cD_{1D}(m_{e,k} + (k-1)\frac\pi{4\varepsilon}) f\,,
\end{equation}
and the claim follows.
\end{proof}
\begin{rem}
Notice that for $k=1$, the one-dimensional Dirac operator in \eqref{eq:k-mode_Dirac} does not depend on $\eps$, that is
\[
U_1\cC_1(\varepsilon) U_1^{-1}  = \Big(-i\sigma_1 \frac{d}{ds} + \frac{2}{\pi}m\sigma_3\Big)  = \cD_{1D}(2\pi^{-1}m). 
\]
\end{rem}
The next lemma concerns the off-diagonal operators $\Pi_j\cE_0(\eps,m)\Pi_k$ for $j,k\in \N$ and $j\neq k$.
\begin{lem} Let $k,j \geq 1$ such that $j \neq k$. The operator $\Pi_j \cE_0(\varepsilon,m) \Pi_k$ satisfies for all $u \in \dom{\cE_0(\eps,m)}$:
\[
	\Pi_j \cE_0(\varepsilon,m) \Pi_ku = m \Pi_j \sigma_3 \Pi_ku.
\]
Hence $\Pi_j \cE_0(\varepsilon,m) \Pi_k$ can be extended uniquely into a bounded operator in $L^2(\Str,\C^2)$ with same operator norm.
\label{lem:sandwich}
\end{lem}
\begin{proof}[Proof of Lemma \ref{lem:sandwich}] Let $v \in \dom{\cE_0(\varepsilon,m)}$ and $k,j \geq 1$ such that $k\neq j$. Set $\Pi_k v = f^+ u_k^+ + f^- u_k^- \in \dom{\cE_0(\eps,m)}$, there holds
\begin{align*}
	\Pi_j \cE_0(\varepsilon,m) \Pi_k v =&\ \Pi_j \Big((-i(f^+)')u_k^- + (-i(f^-)')u_k^+ + k \frac{\pi}{4\varepsilon}(f^+ u_k^+ - f^- u_k^-)\Big) \\&+ m\Pi_j \sigma_3 \Pi_k v \\=&\ m\Pi_j \sigma_3 \Pi_k v.
\end{align*}
As $m\Pi_j \sigma_3 \Pi_k$ is a bounded operator in $L^2(\Str,\C^2)$ and $\dom{\cE_0(\varepsilon,m)}$ is dense in $L^2(\Str,\C^2)$ we deduce that $\Pi_j \cE_0(\varepsilon,m) \Pi_k$ can be extended uniquely
into a bounded operator in $L^2(\Str,\C^2)$ and this operator acts as $m\Pi_j\sigma_3\Pi_k$.
\end{proof}
\begin{prop}\label{prop:invertible_block} Let $\cC_1^\perp(\eps,m)$ be the operator defined in \eqref{eqn:C1}. The operator $\cC_1^\perp(\eps,m)-i$ acting in $\Pi_1^\perp L^2(\Str,\C^2)$ is boundedly invertible and there exists $C >0$ and $\varepsilon_0 >0$ such that for all $\eps \in (0,\varepsilon_0)$ there holds
\[
	\|(\cC_1^\perp(\varepsilon,m) -i)^{-1}\|_{\mathcal{B}(\Pi_1^\perp L^2(\Str,\C^2))} \leq C \varepsilon.
\]
\end{prop}
\begin{rem}
In Proposition \ref{prop:invertible_block}, we used the notation $\cB(\cH)$ which for a complex Hilbert-space $\cH$ stands for the space of bounded operators on $\cH$. Similarly, if $\cH_1$ and $\cH_2$ are two complex Hilbert spaces $\cB(\cH_1,\cH_2)$ denotes the set of bounded operators from $\cH_1$ to $\cH_2$.
\end{rem}
\begin{proof}[Proof of Proposition \ref{prop:invertible_block}] First, remark that $\cC_1^\perp(\varepsilon,m)$ is a self-adjoint operator when acting in $\Pi_1^\perp L^2(\Str,\C^2)$ with domain $\Pi_1^\perp\dom{\cE_0(\eps,m)}$. Hence, the operator $\cC_1^\perp(\varepsilon,m)-i$ is boundedly invertible in $\Pi_1^\perp L^2(\Str,\C^2)$. Second, observe that on $\Pi_1^\perp L^2(\Str,\C^2)$ there holds
\begin{align*}
	\cC_1^\perp(\eps,m) =&\ \big(\sum_{j\geq 2} \Pi_j\big) \cC(\varepsilon,m) \big(\sum_{k\geq 2} \Pi_k\big)\\ =&\ \sum_{j\geq 2} (\Pi_j\cC(\varepsilon,m)\Pi_j ) + \sum_{\substack{j,k\geq 2\\j\neq k}} (\Pi_j\cC(\varepsilon,m)\Pi_k)\\
	=& \ \underset{:= \cG(\eps,m)}{\underbrace{\sum_{j\geq 2} (\Pi_j\cC(\varepsilon,m)\Pi_j )}} + m \underset{:=B}{\underbrace{\sum_{\substack{j,k\geq 2\\j\neq k}} \Pi_j\sigma_3\Pi_k}},
\end{align*}
where we have used Lemma \ref{lem:sandwich} in the last equation, observing that 
\[
 \Pi_j\cC(\varepsilon,m)\Pi_k=\Pi_j\cE_0(\varepsilon,m)\Pi_k\,,\qquad\mbox{if $j\neq k$.}
\]
Remark that as defined, the operator $\cG(\varepsilon,m)$ is self-adjoint and $B\in \cB(\Pi_1^\perp L^2(\Str,\C^2))$.
Indeed, we have
\[
	\sum_{\substack{j,k\geq 2\\j\neq k}} \Pi_j\sigma_3\Pi_k = \sum_{j\geq 2}\Pi_j \sigma_3 \sum_{\substack{k\geq 2\\k\neq j}}\Pi_k = \sum_{j\geq 2} \Pi_j\sigma_3 (\Pi_1^\perp -  \Pi_j) = \Pi_1^\perp\sigma_3\Pi_1^\perp - \sum_{j\geq2}\Pi_j\sigma_3\Pi_j. 
\]
Now, the first term on the right-hand side is a bounded operator in $\Pi_1^\perp L^2(\Str,\C^2)$ while for the second we can argue as follows. Let $u\in \Pi_1^\perp L^2(\Str,\C^2)$, there holds
\begin{align}
	\nonumber\Big\|\sum_{j\geq2}\Pi_j\sigma_3\Pi_ju\Big\|_{\Pi_1^\perp L^2(\Str,\C^2)}^2 = \sum_{j\geq2}\|\Pi_j\sigma_3\Pi_ju\|_{\Pi_1^\perp L^2(\Str,\C^2)}^2 &\leq \sum_{j\geq 2}\|\Pi_j u\|_{\Pi_1^\perp L^2(\Str,\C^2)}^2 \\&= \|u\|_{\Pi_1^\perp L^2(\Str,\C^2)}^2.\label{eqn:boundB}
\end{align}
Moreover, we have
\begin{equation}\label{eq:decomposumop}
	(\cC_1^\perp(\varepsilon,m)-i)^{-1} = (\cG(\eps,m) - i)^{-1}\big(1 + mB(\cG(\eps,m)-i)^{-1}\big)^{-1}.
\end{equation}
Now, we need to estimate $\|(\cG(\varepsilon,m) - i)^{-1}\|_{\cB(\Pi_1^\perp L^2(\Str,\C^2))} = \dist(i, Sp(\cG(\varepsilon,m)))^{-1}$. Recall that by construction we have
\[
	\cG(\eps,m) = \bigoplus_{k\geq 2} \cC_k(\eps,m),
\]
see \cite[p.~268]{RS4} 
for the definition of the direct sum of self-adjoint operators. In particular, by \cite[Thm.~XIII.85]{RS4}, there holds
\[
	\Sp(\cG(\varepsilon,m)) = \bigcup_{k \geq 2} \Sp(\cC_k(\varepsilon,m)) = 
	\big(-\infty,-\frac{\pi}{4\varepsilon}\big]\ 
	\cup \ \big[\frac\pi{4\varepsilon},+\infty\big).
\]
Indeed, thanks to Lemma \ref{lem:uniequiv1D} for all $k\geq 2$ there holds
\[
 \Sp(\cC_k(\varepsilon,m)) = 
 \big(-\infty, -(k-1)\frac\pi{4\eps} -m_{e,k}\big] \cup 
 \big[(k-1)\frac\pi{4\eps} +m_{e,k},+\infty\big)
\]
and for all $k \geq 2$ we have
\[
	\inf_{k\geq 2}
	\left\{m_{e,k} + (k-1)\frac{\pi}{4\varepsilon}\right\} 
	= \frac{\pi}{4\varepsilon}.
\]
Hence, we get $\dist(i,\Sp(\cG(\varepsilon,m))) = \sqrt{ 1 + \frac{\pi^2}{16 \varepsilon^2}}$ and we obtain
\[
	\|(\cG(\varepsilon,m) -i)^{-1}\|_{B(\Pi_1^\perp L^2(\Str,\C^2))} = \frac1{\sqrt{ 1 + \frac{\pi^2}{16 \varepsilon^2}}}.
\]
In particular, we get
\begin{equation}\label{eq:T_resolventnorm}
	\|(\cG(\varepsilon,m) -i)^{-1}\|_{B(\Pi_1^\perp L^2(\Str,\C^2))} = \frac4{\pi}\eps + \mathcal{O}(\eps^3),\quad \text{when }\eps\to 0.
\end{equation}
Next, remark that by \eqref{eqn:boundB} there holds $\|B\|_{\cB(\Pi_1^\perp L^2(\Str,\C^2))} \leq 2$ and using a Neumann series, we arrive at
\begin{equation}\label{eq:neuserie}
	\|\big(1 + mB(\cG(\eps,m)-i)^{-1}\big)^{-1}\|_{\cB(\Pi_1^\perp L^2(\Str,\C^2))} = 1 + \mathcal{O}(\eps),\quad\text{when } \eps \to 0.
\end{equation}
Finally, combining \eqref{eq:T_resolventnorm} and \eqref{eq:neuserie}, \eqref{eq:decomposumop} yields
\[
	\|(\cC_1^\perp - i)^{-1}\|_{\cB(\Pi_1^\perp L^2(\Str,\C^2))} \leq \frac{4}\pi\eps + \mathcal{O}(\eps^2),\quad \text{when }\eps\to0.
\]
It concludes the proof of Proposition \ref{prop:invertible_block}.
\end{proof}
\subsubsection{Proof of Theorem \ref{thm:thinguide} 
in the case of the straight strip}

In this paragraph we prove Theorem \ref{thm:thinguide} in the special case of a straight strip 
but first, we need the next proposition 
whose proof is a direct application of block operator matrices theory.
\begin{prop} \label{prop:block_inverse}Recall that $U$ is the unitary map defined in \eqref{eqn:unitthin}. There holds
\[
	U(\cC(\varepsilon,m)- i)^{-1}U^{-1} = \begin{pmatrix} C_{1,1}(\varepsilon,m) & C_{1,2}(\varepsilon,m)\\ C_{2,1}(\varepsilon,m) & C_{2,2}(\varepsilon,m)
	\end{pmatrix}
\]
where
\[
\begin{split}
	C_{1,1}(\varepsilon,m) &:= (C_{1}(\varepsilon,m) - i)^{-1} \\&\qquad+ m^2(C_{1}(\varepsilon,m) - i)^{-1}\Pi_1\sigma_3\Pi_1^\perp\cS(i)^{-1}\Pi_1^\perp\sigma_3\Pi_1(C_{1}(\varepsilon,m) - i)^{-1},
	\\
	C_{1,2}(\varepsilon,m) & :=  -m(\cC_1(m,\eps)-i)^{-1}\Pi_1\sigma_3\Pi_1^\perp\cS(i)^{-1},
	\\
	C_{2,1}(\varepsilon,m) & :=  - m\cS(i)^{-1}\Pi_1^\perp\sigma_3\Pi_1(\cC_1(\eps,m)-i)^{-1},
	\\
	C_{2,2}(\varepsilon,m) & :=  \cS(i)^{-1}.
	\end{split}
\]
Here, $\cS(i)$ denotes the Schur complement:
	\begin{equation}
		\cS(i) :=\cC_1^\perp(\eps,m) - i - m^2\Pi_1^\perp\sigma_3\Pi_1(\cC_1(\eps,m)-i)^{-1}\Pi_1\sigma_3\Pi_1^\perp.
		\label{eqn:defSchur}
	\end{equation}
\end{prop}
\begin{proof}[Proof of Proposition \ref{prop:block_inverse}]
According to the notation of \cite[Thm.~2.3.3]{Tretter}, 
we set
\[A := \cC_1(\eps,m),\quad B := m\Pi_1\sigma_3\Pi_1^\perp,\quad C := m\Pi_1^\perp \sigma_3 \Pi_1, \quad D := \cC_1^\perp(\eps,m),
\]
where we have used Lemma \ref{lem:sandwich} to rewrite the operators $B$ and $C$.

Now, we check all the hypothesis of \cite[Thm.~2.3.3]{Tretter}:
\begin{itemize}
	\item $\dom{A} = \Pi_1 \dom{\cE_0(\eps,m)} \subset \dom{C} = \Pi_1 L^2(\Str,\C^2)$,
	\item $A$ is self-adjoint as an operator acting in $\Pi_1L^2(\Str,\C^2)$ thus $i \notin \Sp(A)$,
	\item as $A$ is self-adjoint and $B$ is bounded, the operator $(A-i)^{-1}B$ is bounded in $\Pi_1^\perp L^2(\Str,\C^2)$,
	\item the operator $\cS(i)$ is closed because $D$ is self-adjoint and the operator $\Pi_1^\perp\sigma_3\Pi_1(A-i)^{-1}\Pi_1\sigma_3\Pi_1^\perp \in \cB(\Pi_1^\perp L^2(\Str,\C^2))$ (hence both are closed).
	\end{itemize}
Thus, \cite[Thm. 2.3.3]{Tretter} yields
\[
	U(\cC(\varepsilon,m)-i)^{-1}U^{-1} = \begin{pmatrix}\cC_{1,1}(\eps,m)& \cC_{1,2}(\eps,m)\\\cC_{2,1}(\eps,m) & \cC_{2,2}(\eps,m)\end{pmatrix}
\]
with
\begin{align*}
	\cC_{1,1}(\eps,m) &:=  (A-i)^{-1}\Big(\mathds{1} + B \cS(i)^{-1}C(A-i)^{-1}\Big)\\
	\cC_{1,2}(\eps,m) &:=  -(A-i)^{-1}B\cS(i)^{-1}\\
	\cC_{2,1}(\eps,m) &:=  - \cS(i)^{-1}C(A-i)^{-1}\\
	\cC_{2,2}(\eps,m) &:=  \cS(i)^{-1}.
\end{align*}
This finishes the proof.
\end{proof}
%
We are now in a good position to prove \eqref{eq:resolvent_convergence} in Theorem \ref{thm:thinguide} for the straight waveguide.
\begin{prop}\label{prop:thinguide_flat}
There exists a unitary map $V$ such that $V : L^2(\Str,\C^2) \to L^2(\R,\C^2)\oplus \Pi_1^\perp L^2(\Str,\C^2) $ and there holds
\[
	V\big(\cE_0(\varepsilon,m)- \frac{\pi}{4\varepsilon}(P^+ - P^-)-i\big)^{-1}V^{-1} = \big(\cD_{1D}(2\pi^{-1} m)-i\big)^{-1} \oplus 0 + \mathcal{O}(\varepsilon)\,,
\]
in the operator norm, where $P^\pm$ are the projectors defined in \eqref{eq:a_projectors}.
\end{prop}
\begin{proof}[Proof of Proposition \ref{prop:thinguide_flat}]
The proof is performed in three steps. In the first two steps we estimate the norm of the  bounded operators $(\cC_1(\eps,m)-i)^{-1}$ and the Schur complement $\cS(i)^{-1}$ (defined in \eqref{eqn:defSchur}). In the last step, we use Proposition \ref{prop:block_inverse} to obtain an asymptotic expansion of the operator $U(\cC(\eps,m) - i)^{-1}U^{-1}$.
\medskip

\paragraph{\bf{Step 1}.} Thanks to Lemma \ref{lem:uniequiv1D}, we know that $\Sp(\cC_1(\eps,m)) = (-\infty,-\frac{2}\pi m] \cup [\frac2\pi m,+\infty)$. In particular, there holds
\begin{equation}\label{eq:normsanseps}
	\|(\cC_1(\eps,m) -i)^{-1}\|_{\cB(\Pi_1L^2(\Str,\C^2))} = \frac{1}{\dist(i,\Sp(\cC_1(\eps,m)))} = \frac1{\sqrt{1+\frac4{\pi^2}m^2}}.
\end{equation}
\medskip
\paragraph{\bf{Step 2}.}
Remark that there holds
\[
	\cS(i)^{-1} = \big(\mathds{1} - m^2(\cC_1^\perp(\eps,m) -i)^{-1}\Pi_1^\perp\sigma_3\Pi_1(\cC_1(\eps,m)-i)^{-1}\Pi_1\sigma_3\Pi_1^\perp\big)^{-1}(\cC_1^\perp(\eps,m) -i)^{-1}
\]
and in particular, we have
\begin{align*}
	\nonumber&\|(\cC_1^\perp(\eps,m) -i)^{-1}\Pi_1^\perp\sigma_3\Pi_1(\cC_1(\eps,m)-i)^{-1}\Pi_1\sigma_3\Pi_1^\perp\|_{\cB(\Pi_1^\perp L^2(\Str,\C^2))}\\ &\qquad\qquad\qquad\leq \|(\cC_1^\perp(\eps,m) -i)^{-1}\|_{\cB(\Pi_1^\perp L^2(\Str,\C^2))} \|(\cC_1(\eps,m) -i)^{-1}\|_{\cB(\Pi_1L^2(\Str,\C^2))}\nonumber\\ &\qquad\qquad\qquad\leq \frac{C}{\sqrt{1 + \frac4{\pi^2}m^2}} \eps := \tilde{C}\eps,\quad\text{when }\eps \to 0.
\end{align*}
Here, the first inequality is obtained using that $\sigma_3$ is a unitary operator from $L^2(\Str,\C^2)$ onto itself and that $\Pi_1$ and $\Pi_1^\perp$, being orthogonal projectors, are bounded operators with norm smaller than $1$. The second inequality is a consequence of \eqref{eq:normsanseps} and Proposition \ref{prop:invertible_block}.
In particular, using a Neumann series and Proposition \ref{prop:invertible_block}, it yields the existence of $C' > 0$ and $\eps_1 >0$ such that for all $\eps \in(0,\eps_1)$ there holds
\begin{equation}
\label{eqn:borneS2}
	\|\cS(i)^{-1}\|_{B(\Pi_1^\perp L^2(\Str,\C^2))} \leq C' \eps.
\end{equation}
\medskip
\paragraph{\bf{Step 3}.}
Thanks to Proposition \ref{prop:block_inverse} there holds
\[
	U(\cC(\eps,m) - i)^{-1}U^{-1} = \begin{pmatrix}(\cC_1(\eps,m) - i)^{-1} & 0 \\ 0 & 0\end{pmatrix} + \begin{pmatrix}R_{1,1}(\eps,m) & C_{1,2}(\eps,m) \\ C_{2,1}(\eps,m) & C_{2,2}(\eps,m)\end{pmatrix},
\]
where we have set
\[
R_{1,1}(\eps,m) = m^2(C_{1}(\varepsilon,m) - i)^{-1}\Pi_1\sigma_3\Pi_1^\perp\cS(i)^{-1}\Pi_1^\perp\sigma_3\Pi_1(C_{1}(\varepsilon,m) - i)^{-1}.
\]
Now, we examine the norm of each bounded operator appearing in the second block matrix 
on the right-hand side. Remark that by \eqref{eq:normsanseps} and \eqref{eqn:borneS2}, for all $\eps \in (0,\eps_1)$ there holds
\begin{align}\label{eqn:borne1b}
\lefteqn{\|R_{1,1}(\eps,m)\|_{\cB(\Pi_1 L^2(\Str,\C^2))}} 
\nonumber \\ 
&\leq m^2\|(\cC_1(\eps,m) - i)^{-1}\|_{\cB(\Pi_1 L^2(\Str,\C^2))}^2 \|\cS(i)^{-1}\|_{\cB(\Pi_1^\perp L^2(\Str,\C^2))}
\nonumber \\
&\leq m^2\frac{C'}{1 + \frac4{\pi^2}m^2}\eps.
\end{align}
Similarly, for $\eps \in (0,\eps_1)$ there holds
\begin{equation}
\label{eqn:borne2b}
	\|C_{1,2}(\eps,m)\|_{\cB(\Pi_1^\perp L^2(\Str,\C^2),\Pi_1 L^2(\Str,\C^2))} \leq m\frac{\tilde{C}}{\sqrt{1 + \frac{4}{\pi^2}m^2}}\eps
\end{equation}
and
\begin{equation}
\label{eqn:borne3}
	\|C_{2,1}(\eps,m)\|_{\cB(\Pi_1 L^2(\Str,\C^2),\Pi_1^\perp L^2(\Str,\C^2))} \leq m\frac{\tilde{C}}{\sqrt{1 + \frac{4}{\pi^2}m^2}}\eps.
\end{equation}
Gathering \eqref{eqn:borne1b}, \eqref{eqn:borne2b}, \eqref{eqn:borne3} and \eqref{eqn:borneS2} we get
\[
	U(\cC(\eps,m)-i)^{-1}U^{-1} = \begin{pmatrix}(\cC_1(\eps,m) - i)^{-1} & 0 \\ 0 & 0\end{pmatrix} + \mathcal{O}(\eps),\quad\text{when } \eps \to 0.
\]
To conclude, we introduce the unitary map
\[
	V : L^2(\Str,\C^2) \to L^2(\R,\C^2) \oplus \Pi_1^\perp(L^2(\Str,\C^2)),\quad (Vu) := \left(U_1\Pi_1 u,\Pi_1^\perp u\right),
\]
where the unitary map $U_1$ is defined in Lemma \ref{lem:uniequiv1D}. When $\eps \to 0$, there holds
\begin{align*}
	V(\cC(\eps,m)-i)^{-1} V^{-1} &= \begin{pmatrix}U_1(\cC_1(\eps,m) - i)^{-1}U_1^{-1} & 0 \\ 0 & 0\end{pmatrix} + \mathcal{O}(\eps)\\
	& = \begin{pmatrix}(\cD_{1D}(2\pi^{-1}m) -i)^{-1} & 0 \\ 0 & 0\end{pmatrix} + \mathcal{O}(\eps)\\
	& = (\cD_{1D}(2\pi^{-1}m) -i)^{-1} \oplus 0 + \mathcal{O}(\varepsilon).
\end{align*}
\end{proof}
\subsection{Convergence for the curved waveguide}\label{subsec:proofthm3}
This paragraph is devoted to the proof of Theorem \ref{thm:thinguide}. Once again, we use a perturbation argument. We start with a few auxiliary results.

The first lemma deals with the quadratic form for the {\it transverse part} of the operator.

\be\label{eq:1Dquadform}
\begin{split}
	\mathfrak{\tau}_m(u) &:= \|(-i\sigma_2)u'\|^2_{L^2((-1,1),\C^2)} + m\big(\vert u(1)\vert^2 + \vert u(-1)\vert^2\big),\\ \dom{\mathfrak{\tau}_m} &:= \{u=(u_1,u_2)^\top \in H^1((-1,1),\C^2) : u_2(\pm1) = \mp u_1(\pm1)\}\,.
	\end{split}
\ee
Remark that $\tau_m$ is the quadratic form associated with the operator $\cT(0,m)^2 - m^2$, where $\cT(0,m)$ is defined in \eqref{eqn:defA}, as can be seen in \eqref{eq:squaredHk} and below.
\begin{lem} Let $u\in \dom{\mathfrak{\tau}_m}$, there holds
\[
	\mathfrak{\tau}_m(u) \geq E_1(m)\|\pi_1 u\|_{L^2((-1,1),\C^2)}^2 + \mathfrak{\tau}_0(\pi_1^\perp u),
\]
where the projector $\pi_1$ is defined in \S \ref{subsec:cvgstr} and where we have set $\pi_1^\perp = \mathds{1}- \pi_1$.
\label{lem:bddbelow1Dform}
\end{lem}
\begin{proof}[Proof of Lemma \ref{lem:bddbelow1Dform}] Let $u\in\dom{\mathfrak{\tau}_m}$, there holds
\begin{align}
	\nonumber \mathfrak{\tau}_m(u)  = \mathfrak{\tau}_m(\pi_1 u + \pi_1^\perp u) & = \mathfrak{\tau}_m(\pi_1 u) + \mathfrak{\tau}_m(\pi_1^\perp u) + 2\Re(\mathfrak{\tau}_m(\pi_1 u, \pi_1^\perp u))\\
	\nonumber&\geq E_1(m) \|\pi_1 u\|_{L^2((-1,1),\C^2)}^2 + \mathfrak{\tau}_0(\pi_1^\perp u)\\&\qquad + 2\Re(\mathfrak{\tau}_m(\pi_1 u, \pi_1^\perp u))\label{eqn:bdbelowbegin1D},
\end{align}
where we have used the min-max principle (Proposition \ref{prop:min-max}) and bounded from below the quadratic form $\tau_m$ by $\tau_0$.
Now, remark that for all $v\in\dom{\mathfrak{\tau}_m}$ there holds
\[
	\mathfrak{\tau}_m(v) = \|(-i\sigma_2)v' +m \sigma_3v\|_{L^2(\Str,\C^2)}^2 - m^2\|v\|_{L^2(\Str,\C^2)}^2.
\]
In particular, for the associated sesquilinear form it gives
\begin{align*}
	\mathfrak{\tau}_m(\pi_1v,\pi_1^\perp v) &= \langle\big((-i\sigma_2)\frac{d}{dt}  +m \sigma_3\big)\pi_1 v,  \big((-i\sigma_2)\frac{d}{dt}  +m \sigma_3\big)\pi_1^\perp v\rangle_{L^2((-1,1),\C^2)}\\ &\qquad- m^2 \langle\pi_1 v,\pi_1^\perp v\rangle_{L^2((-1,1),\C^2)}\\
	& = m\Big(\langle\cT_0\pi_1 v,   \sigma_3\pi_1^\perp v\rangle_{L^2((-1,1),\C^2)} + \langle\sigma_3\pi_1 v,  \cT_0\pi_1^\perp v\rangle_{L^2((-1,1),\C^2)}\Big).
\end{align*}
Now,  remark that
\begin{align*}
	\langle\cT_0\pi_1 v,   \sigma_3\pi_1^\perp v\rangle_{L^2((-1,1),\C^2)} &= \langle\cT_0 v,   \pi_1\sigma_3\pi_1^\perp v\rangle_{L^2((-1,1),\C^2)},\\ \langle\sigma_3\pi_1 v,  \cT_0\pi_1^\perp v\rangle_{L^2((-1,1),\C^2)} &= \langle\pi_1^\perp\sigma_3\pi_1 v,  \cT_0 u\rangle_{L^2((-1,1),\C^2)}.
\end{align*}
If $v \in \dom{\mathfrak{\tau}_m} = \dom{\cT_0}$, then $\pi_1^\perp\sigma_3\pi_1 u \in \dom{\cT_0}$ and as $\cT_0$ is self-adjoint there holds
\begin{align*}
\langle\cT_0 v,   \pi_1\sigma_3\pi_1^\perp v\rangle_{L^2((-1,1),\C^2)} & = \langle v,   \cT_0\pi_1\sigma_3\pi_1^\perp v\rangle_{L^2((-1,1),\C^2)}\\
& = -\langle \pi_1^\perp\sigma_3\pi_1v, \cT_0 v\rangle_{L^2((-1,1),\C^2)},
\end{align*}
where we have used that $\cT_0$ commutes with $\pi_1$ and $\pi_1^\perp$ and that $\sigma_2$ anti-commutes with~$\sigma_3$.
In particular, we obtain that $\mathfrak{\tau}_m(\pi_1u,\pi_1^\perp u) = 0$ which combined with equation \eqref{eqn:bdbelowbegin1D} yields
\[
	\mathfrak{\tau}_m(u) \geq E_1(m) \|\pi_1 u\|_{L^2((-1,1),\C^2)}^2 + \mathfrak{\tau}_0(\pi_1^\perp u),
\]
which is precisely Lemma \ref{lem:bddbelow1Dform}.
\end{proof}
\begin{rem}
Observe that the quadratic form $\mathfrak{\tau}_m$ in \eqref{eq:1Dquadform} is {\it a priori} defined for functions of $t\in(-1,1)$. With an abuse of notation, in what follows we extend it to functions defined on the strip $\Str$ acting only on the transverse variable. More precisely, there holds $u(s,\cdot)\in\dom{\tau_m}$ for a.e. $s\in\R$, if  $u \in \dom{\cE_0(\eps,m)}$.
\end{rem}
We now state a technical result, of crucial importance in the proof of Theorem \ref{thm:thinguide}, whose proof is postponed to the Appendix \ref{technical} for simplicity.
\begin{lem}\label{lem:difficile} Let $u \in \dom{\cE_0(\eps,m)}$, there exists $\eps_0 >0$ and $K>0$ such that for all $\eps\in (0,\varepsilon_0)$ we have
\[
	\|(-i\sigma_1)\partial_s u + m \sigma_3 u\|_{L^2(\Str,\C^2)} \leq  \|\cC(\eps,m) u \|_{L^2(\Str,\C^2)} + K\|u\|_{L^2(\Str,\C^2)},
\]
where the operator $\mathcal{C}(\eps,m)$ is defined in \eqref{eqn:defopC}.
\end{lem}

We are now in a good position to prove Theorem \ref{thm:thinguide}.
\begin{proof}[Proof of Theorem \ref{thm:thinguide}]
Let us set
\[
	\cC_\Gamma(\eps,m) := \cE_\Gamma(\eps,m) - \frac\pi{4\eps}(P^+ - P^-).
\]
and remark that
\[
	\cC_\Gamma(\eps,m) = \cC(\eps,m) + V(\eps),
\]
where $\cC(\eps,m)$ is defined in \eqref{eqn:defopC} and the symmetric operator $V(\eps)$ is defined in \eqref{eqn:defpertur}.

Consider the operator
\[
	(\cC_\Gamma(\eps,m)-i)^{-1} = (\cC(\eps,m) -i + V(\eps))^{-1} = (\cC(\eps,m) - i)^{-1}\big(\mathds{1} + V(\eps)(\cC(\eps,m)-i)^{-1}\big)^{-1}.
\]
We claim that there exists $\eps_0 > 0$ and $K'>0$ such that for all $\eps \in (0,\eps_0)$ there holds
\[
	\|V(\eps)(\cC(\eps,m)-i)^{-1}\|_{\cB(L^2(\Str,\C^2))} \leq K' \eps.
\]
Indeed, for $u \in L^2(\Str,\C^2)$, there holds
\begin{align*}
	\|V(\eps)(\cC(\eps,m)-i)^{-1} u\|_{L^2(\Str,\C^2)} \leq& \frac{\eps \|\kappa\|_{L^\infty(\R)}}{1-\eps \|\kappa\|_{L^\infty(\R)}} \|(-i\sigma_1)\partial_s (\cC(\eps,m)-i)^{-1}u\|_{L^2(\Str,\C^2)} \\&\quad\quad+ \frac{\eps \|\kappa'\|_{L^\infty(\R)}}{2(1-\eps\|\kappa\|_{L^\infty(\R)})^2} \|(\cC(\eps,m)-i)^{-1}\|_{\cB(L^2(\Str,\C^2))}\|u\|_{L^2(\Str,\C^2)}.
\end{align*}
One remarks that
\begin{align*}
	\|(-i\sigma_1)\partial_s(\cC(\eps,m)-i)^{-1}u\|_{L^2(\Str,\C^2)} & = \|\big((-i\sigma_1)\partial_s + m\sigma_3 - m \sigma_3\big)(\cC(\eps,m)-i)^{-1}u\|_{L^2(\Str,\C^2)}\\
	&\leq \|\big((-i\sigma_1)\partial_s + m\sigma_3\big)(\cC(\eps,m)-i)^{-1}u\|_{L^2(\Str,\C^2)} \\&\qquad + m \|(\cC(\eps,m)-i)^{-1}u\|_{L^2(\Str,\C^2)}.
\end{align*}
Hence, by Lemma \ref{lem:difficile}, there exists $K>0$ and $\eps_1 >0$ such that for all $\eps \in (0,\eps_1)$ there holds:
\begin{align*}
	\|(-i\sigma_1)\partial_s(\cC(\eps,m)-i)^{-1}u\|_{L^2(\Str,\C^2)} & \leq \|\cC(\eps,m)(\cC(\eps,m)-i)^{-1}u\|_{L^2(\Str,\C^2)} \\&\qquad + (m+K) \|(\cC(\eps,m)-i)^{-1}u\|_{L^2(\Str,\C^2)}\\
	& \leq\|(\cC(\eps,m)-i)(\cC(\eps,m)-i)^{-1}u\|_{L^2(\Str,\C^2)}\\&\qquad + (1+m+K) \|(\cC(\eps,m)-i)^{-1}u\|_{L^2(\Str,\C^2)}\\
	&\leq \|u\|_{L^2(\Str,\C^2)} \\&\qquad+ (m+1+K)\|(\cC(\eps,m)-i)^{-1}\|_{\cB(L^2(\Str,\C^2))} \|u\|_{L^2(\Str,\C^2)}.
\end{align*}
Remarking that
\begin{equation}\label{eqn:bornerescC}\|(\cC(\eps,m)-i)^{-1}\|_{\cB(L^2(\Str,\C^2))} = \dist(i,Sp(\cC(\eps,m)))^{-1} \leq 1\end{equation} we obtain that there exists $\eps_0 \in(0, \eps_1)$ such that
\[
	\|V(\eps)(\cC(\eps,m)-i)^{-1} u\|_{L^2(\Str,\C^2)} \leq K' \eps \|u\|_{L^2(\Str,\C^2)},
\]
for some constant $K'>0$. Thus, developing in Neumann series and using \eqref{eqn:bornerescC} we get
\[
	(\cC_\Gamma(\eps,m) -i)^{-1} = (\cC(\eps,m) -i)^{-1} + \mathcal{O}(\eps)
\]
and the theorem is proved applying Proposition \ref{prop:thinguide_flat}.
\end{proof}
\section{Non-relativistic limit}\label{sec:nonrel}

This section is devoted to the proof of Proposition \ref{cor:evnumber} and Theorem \ref{thm:nonrelativistic}. In the sequel we will assume $\varepsilon>0$ to be fixed, as we are only interested in the regime $m\to+\infty$. We start by proving Proposition \ref{cor:evnumber} before turning to the proof of Theorem \ref{thm:nonrelativistic}. 
\subsection{Proof of Proposition \ref{cor:evnumber}}
Our starting point is the expression of the quadratic form associated with the operator $\cD_\Gamma(\varepsilon,m)^2$, that can be computed arguing as in \cite[Prop. 14]{Lotoreichik-Ourmieres_2019}.

\begin{lem}
Given $u\in\dom{\cD_\Gamma(\eps,m)}$, there holds
\[
\|\cD_\Gamma(\eps,m) u\|_{L^2(\Omega_\varepsilon,\C^2)}^2=\|\nabla u\|_{L^2(\Omega_\varepsilon,\C^2)}^2+m^2\| u\|_{L^2(\Omega_\varepsilon,\C^2)}^{2}+\int_{\partial\Omega_\eps}(m-\frac{\kappa_\eps}2) \vert u\vert^2 ds\,,
\]
where $\kappa_\eps$ is the signed curvature of the boundary $\partial\Omega_\eps$ with respect to the outer normal~$\nu_\eps$.
\label{lem:quadformsquare}
\end{lem}
Let us introduce the quadratic forms
\[
q_m(u) := \|\cD_\Gamma(\eps,m) u\|_{L^2(\Omega_\eps,\C^2)}^2 - m^2\|u\|_{L^2(\Omega_\eps,\C^2)}^2,\quad \dom{q_m} := \dom{\cD_\Gamma},
\]
and
\[
q_\infty(u) := \|\nabla u\|_{L^2(\Omega_\eps,\C^4)}^2,\quad \dom{q_\infty} :=	 H_0^1(\Omega_\eps,\C^2).
\] of the Dirichlet Laplacian $\cL_\Gamma(\eps)$ defined in~\eqref{eq:Laplacian}.
In the following we shall consider the min-max values of the forms above, as introduced in Definition~\ref{def:minmax}.
We are now in a good position to prove Proposition~\ref{cor:evnumber}.
\begin{proof}[Proof of Proposition~\ref{cor:evnumber}]
Observe that $\dom{q_\infty}\subset\dom{\cD_\Gamma(\eps,m)}$ and that by Lemma \ref{lem:quadformsquare} if $u\in\dom{q_\infty}$ we have $q_\infty(u)=q_m(u)$. Then, by Proposition \ref{prop:min-max}, we immediately get for all $j \in \N$:
\be\label{eq:minmaxineq}
\mu_j(q_m)\leq\mu_j(q_\infty).
\ee
Recall that by Theorem \ref{prop:sa-essspec}, $\eps^{-2}{E_1(m\eps)}$ is the bottom of the essential spectrum of $\cD_\Gamma(\eps,m)^2-m^2$. Now, fix $j_0 \in \N$ with $j_0 < N_\Gamma + 1$ (with the convention that $N_\Gamma +1 = +\infty$ if $N_\Gamma = +\infty$). Then, by Proposition \ref{prop:min-max}, 
\ref{itm:pt5} of Proposition \ref{prop:op1Dess} and Proposition \ref{prop:specdirlapl}, we get for all $j\in \{1,\dots,2j_0\}$:
\[
\mu_j(q_m)-\frac{E_1(m\eps)}{\eps^2}\leq\mu_j(q_\infty)-\frac{E_1(m\eps)}{\eps^2}\leq \underbrace{\mu_{2j_0}(q_\infty)-\frac{\pi^2}{4\eps^2}}_{<0}+\frac{C}{m}\,,
\]
for some constant $C>0$. Then the claim follows taking $m$ large enough.
\end{proof}
\subsection{Finite waveguides} In our proof of Theorem \ref{thm:nonrelativistic}, we need to investigate the min-max values of quadratic forms in finite waveguides. To this aim, for $R> 0$ we split the waveguide $\Omega_\eps$ into the following three domains
\begin{align*}
	\Omega^R_\eps & := \{\gamma(s)+\eps t\nu(s) : \vert s\vert<R\,,\ t\in(-1,1) \},\\ 
	\Omega^{R,\pm}_\eps & := \{\gamma(s)+\eps t\nu(s) :  \pm s>R\,,\ t\in(-1,1) \}, 
\end{align*}
and consider the following four forms:
\begin{align*}
	q_\infty^R(u) &:= \|\nabla u\|_{L^2(\Omega_\eps^R,\C^4)}^2,\\
	\dom{q_\infty^R} &:=  H_0^1(\Omega_\eps^R,\C^2),\\
	q_m^R(u) &:= \|\cD_\Gamma(\eps,m) u\|_{L^2(\Omega_\eps^R,\C^2)}^2 - m^2\|u\|_{L^2(\Omega_\eps^R,\C^2)}^2,\\
	\dom{q_m^R} &:=  \big\{u \in H^1(\Omega_\eps^R,\C^2) : -i\sigma_3\sigma\cdot\nu_\eps u = u \text{ on } \partial\Omega_\eps^R\cap \partial\Omega_\eps,\\&\qquad u=0 \text{ on } \partial\Omega_\eps^R \setminus \partial\Omega_\eps\big\},\\
	q_m^{R,\pm}(u)&:= \|\cD_\Gamma(\eps,m) u\|_{L^2(\Omega_\eps^{R,\pm},\C^2)}^2 - m^2\|u\|_{L^2(\Omega_\eps^{R,\pm},\C^2)}^2,\\\qquad
	\dom{q_m^{R,\pm}} & := \big\{u \in H^1(\Omega_\eps^{R,\pm},\C^2) : -i\sigma_3\sigma\cdot\nu_\eps u = u \text{ on } \partial\Omega_\eps^{R,\pm}\cap \partial\Omega_\eps,\\&\qquad u=0 \text{ on } \partial\Omega_\eps^{R,\pm} \setminus \partial\Omega_\eps\big\}.
\end{align*}

In the following we shall consider the min-max values of the above forms as introduced in Definition \ref{def:minmax}.

The same compactness argument as in \cite[Prop. 2.1]{Arrizabalaga-LeTreust-Raymond17} allows to prove the following local convergence result, whose proof is omitted.
\begin{lem}\label{lem:localconv}
For all $R>0$ and $j\in\N$, there holds
\[
\lim_{m\to+\infty}\mu_j(q^R_m)=\mu_j(q^R_\infty)\,.
\]
\end{lem}
For further use, we need the following lemma which is proved using the well-known IMS formula.
\begin{lem}\label{lem:dirconv}
For all $j\in \mathbb{N}$ there holds
\[
\lim_{R\to+\infty}\mu_j(q^R_\infty)=\mu_j(q_\infty)\,.
\]
\end{lem}
\begin{proof}[Proof of Lemma \ref{lem:dirconv}]
Fix $j\in\mathbb{N}$ and observe that thanks to a Dirichlet bracketing argument one gets $\mu_j(q^R_\infty)\geq\mu_j(q_\infty)$, for all $R>0$. Then
\be\label{eq:lowboundlimit}
\liminf_{R\to\infty}\mu_j(q^R_\infty)\geq\mu_j(q_\infty)\,.
\ee
Now, we need to prove the opposite inequality
\be\label{eq:limsup}
\limsup_{R\to\infty} \mu_j(q^R_\infty)\leq \mu_j(q_\infty)\,.
\ee

Take a cut-off function $\theta\in C^\infty_0(\R)$ such that $0\leq\theta\leq1$, $\theta(s)= 1$ for $\vert s\vert\leq\frac{1}{2}$ and $\theta(s) = 0$ for $\vert s\vert\geq1$. Given $R>0$, define
\[
\theta_R(s):=\theta(R^{-1}s)\,,\qquad s\in\R\,.
\]
We introduce
\[
\chi_R:=(U_1)^{-1}\theta_R,
\]
where $U_1$ is the unitary map \eqref{eq:U_1}. For further use, we compute $\nabla\chi_R$. 

Since $\chi_R(\gamma(s)+\eps t \nu(s))=\theta(R^{-1}s)$, we get
\be\label{eq:chiderivatives}
\left\{\begin{aligned}
	\partial_s\chi_R&=\gamma'_1(1-\eps t\kappa)\partial_1\chi_R+\gamma'_2(1-\eps t\kappa)\partial_2\chi_R=R^{-1}\theta'(R^{-1}s)\,, \\
	\partial_t\chi_R&=\eps\nu_1\partial_1\chi_R+\eps\nu_2\partial_2\chi_R=0\,,
\end{aligned}\right.
\ee
where $\gamma'=(\gamma_1,\gamma_2)^\top$ and $\nu=(\nu_1,\nu_2)^\top = (-\gamma'_2,\gamma'_1)^\top$.  Then \eqref{eq:chiderivatives} can be rewritten as
\be\label{eq:matrixgradchi}
\begin{pmatrix}
\gamma'_1(1-\eps t\kappa) & \gamma'_2(1-\eps t\kappa) \\
-\eps\gamma'_2 & \eps\gamma'_1
\end{pmatrix}
\begin{pmatrix}
\partial_1\chi_R \\
\partial_2\chi_R
\end{pmatrix}
=
\begin{pmatrix}
R^{-1}\theta'(R^{-1}s) \\ 0
\end{pmatrix}
\,,
\ee
so that, inverting the matrix in \eqref{eq:matrixgradchi} and after straightforward computations one finds for $x = \gamma(s) + \eps t\nu(s)$ :
\be\label{eq:gradchi}
\nabla\chi_R(x)=\nabla\chi_R(\gamma(s)+\eps t \nu(s))=\frac{\theta'(R^{-1}s)}{R(1-\eps t\kappa(s))}\gamma'(s).
\ee
Take $u =(u_1,u_2)^\top\in\dom{q_\infty}$. As chosen, we have $\chi_Ru\in\dom{q_\infty^R}$.
Thus, we find
\be\label{eq:dirformlowbound}
q_\infty(\chi_{R}u)=q^R_\infty(\chi_Ru).
\ee
On the other hand, we have
\be\label{eq:formcutspinor}
\begin{split}
q_\infty(\chi_Ru)&=\sum_{k=1}^2\Big(\underset{:= a_k}{\underbrace{\|\chi_R \nabla u_k\|_{L^2(\Omega_\eps,\C^2)}^2}} + \underset{:= b_k}{\underbrace{\|u_k \nabla \chi_R\|_{L^2(\Omega_\eps,\C^2)}^2}} \\&\qquad+ \underset{:= c_k}{\underbrace{2 \Re \big(\langle\chi_R \nabla u_k,u_k\nabla\chi_R\rangle_{L^2(\Omega_\eps,\C^2}\big)}}\Big).
\end{split}
\ee
Let $k\in \{1,2\}$, we get
\[
a_k\leq \|\nabla u_k\|^2_{L^2(\Omega_\eps,\C^2)}.
\]
By \eqref{eq:gradchi} the second term $b_k$ can be estimated as
\[
b_k\leq \frac{\Vert \theta'\Vert^2_{L^\infty(\R)}}{R^2(1-\eps \|\kappa\|_{L^\infty(\R)})^2}\Vert u_k\Vert^2_{L^2(\Omega_\eps)}\,.
\]
Similarly, we obtain
\[
\begin{split}
c_k &\leq 2 \|(\nabla\chi_R) u_k\|_{L^2(\Omega_\eps,\C^2)}\|\chi_R \nabla u_k\|_{L^2(\Omega_\eps,\C^2)}\\
&\leq  \frac{2 \Vert \theta'\Vert_{L^\infty(\R)}}{R(1-\eps \Vert\kappa\Vert_{L^\infty(\R)})}\|u_k\|_{L^2(\Omega_\eps)}\|\nabla u_k\|_{L^2(\Omega_\eps,\C^2)} \\
&\leq  \frac{\Vert \theta'\Vert_{L^\infty(\R)}}{R(1-\eps \Vert\kappa\Vert_{L^\infty(\R)})}(\|\nabla u_k\|_{L^2(\Omega_\eps,\C^2)}^2 + \|u_k\|_{L^2(\Omega_\eps)}^2).
\end{split}
\]

Combining the above estimates with \eqref{eq:dirformlowbound} and \eqref{eq:formcutspinor}, we obtain that there exists $C > 0$ such that for all $R>0$ there holds
\be\label{eq:mu_jineq}
\begin{split}
q_\infty^R(\chi_Ru) \leq \Big(1+\frac{C}R\Big)q_\infty(u) + \frac{C}R\|u\|_{L^2(\Omega_\eps,\C^2)}^2.
\end{split}
\ee
Now, by Definition \ref{def:minmax}, for $\eta > 0$, there exists $W_\eta \subset\dom{q_\infty}$ a $j$-th dimensional vector space such that
\begin{equation}\label{eqn:minsequence}
	\mu_j(q_\infty)\leq \sup_{u\in W_\eta\setminus\{0\}} \frac{q_\infty(u)}{\|u\|_{L^2(\Omega_\eps,\C^2)}^2} \leq \mu_j(q_\infty) + \eta.
\end{equation}
Remark that if $(u_1^\eta,\dots,u_j^\eta)$ is an orthonormal basis of $W_\eta$, then there exists $R_0 := R_0(\eta) > 0$ such that for all $R>R_0$ the family $(\chi_R u_1^\eta,\dots,\chi_Ru_j^\eta)$ is a basis in $L^2(\Omega_\eps^R,\C^2)$ of the vector space $W_\eta^R := \{\chi_R u : u \in \textrm{span}(u_1^\eta,\dots,u_j^\eta) \}$. Indeed, for all $k,p\in \{1,\dots,j\}$ there holds
\[
	\langle\chi_R u_k^\eta,\chi_R u_p^\eta\rangle_{L^2(\Omega_\eps^R,\C^2)} = \delta_{k,p} -\int_{\Omega_\eps^R}(1-\chi_R^2)\langle u_k^\eta,u_p^\eta\rangle_{\C^2}dx.
\]
Hence, by the dominated convergence theorem, the second term on the right-hand side of the above equation converges to $0$ as $R\to +\infty$ and there exists $R_0 > 0$ such that for all $R>R_0$ there holds $\dim(W_\eta^R) = j$.

Now, pick a $u_\star \in W_\eta\setminus\{0\}$ such that
\[
	\frac{q_\infty^R(\chi_R u_\star)}{\|\chi_R u_\star\|_{L^2(\Omega_\eps^R,\C^2)}^2} = \sup_{u\in W_\eta^R\setminus\{0\}} \frac{q_\infty^R(u)}{\|u\|_{L^2(\Omega_\eps^R,\C^2)}^2}\geq \mu_j(q_\infty^R).
\]

Consequently, as $W_\eta^R\subset \dom{q_\infty^R}$, the min-max principle (Proposition \ref{prop:min-max}), \eqref{eq:mu_jineq} and \eqref{eqn:minsequence} give
\begin{equation}
	\mu_j(q_\infty^R)\frac{\|\chi_R u_\star\|_{L^2(\Omega_\eps^R,\C^2)}^2}{\|u_\star\|_{L^2(\Omega_\eps,\C^2)}^2} \leq (1+\frac{C}R) \frac{q_\infty(u_\star)}{\|u_\star\|_{L^2(\Omega_\eps,\C^2)}^2} + \frac{C}R\leq(1+\frac{C}R)(\mu_j(q_\infty)+\eta) + \frac{C}R.
	\label{eqn:ublimfin}
\end{equation}
Observe that by dominated convergence one also gets $\|\chi_Ru_\star\|_{L^2(\Omega_\eps^R,\C^2)}\rightarrow\|u_\star\|_{L^2(\Omega_\eps,\C^2)}$, as $R\to\infty$. Thus, letting $R\to\infty$ in \eqref{eqn:ublimfin}, we obtain the inequality
\[
\limsup_{R\to\infty}\mu_j(q^R_\infty)\leq\mu_j(q_\infty) + \eta\,.
\]
As this is true for all $\eta > 0$, we get \eqref{eq:limsup} and the proof is concluded.
\end{proof}

We conclude this paragraph with the following lemma.

\begin{lem}\label{lem:extlbRgrand} Let us assume additionally that $\Gamma$ is of class $C^4$, that $\kappa'(s)\to 0$ and $\kappa''(s)\to 0$ when $|s|\to +\infty$ and let $R>0$. For all $u \in \dom{q_m^{R,\pm}}$ there holds
\[
	\mu_1(q_m^{R,\pm}) \geq \frac{E_1(m\eps)}{\eps^2} - \eta^\pm(R),
\]
where $\eta^\pm \geq 0$ does not depend on $m$ and verifies $\eta^\pm(R) \to 0$ when $R\to+\infty$.
\end{lem}
\begin{proof}[Proof of Lemma \ref{lem:extlbRgrand}]
Let $u\in \dom{q_m^{R,\pm}}$ and consider $u_0$ its extension by $0$ to the whole waveguide $\Omega_\eps$. Remark that $u_0\in \dom{q_m}$ and set $v_0 = (U_3U_2U_1)u_0$ where the unitary maps $U_1$, $U_2$ and $U_3$ are defined in \eqref{eq:U_1}, \eqref{eq:U_2} and \eqref{eq:U_3} respectively. By Proposition \ref{prop:quadform2}, and using the min-max principle on the operator acting in the $t$-variable we get
\begin{align*}
	q_m^{R,\pm}(u) = q_m(u_0) \geq \ & \frac{E_1(m\eps)}{\eps^2}\|v_0\|_{L^2(\Str,\C^2)}^2 - \int_{\Str} \frac{\kappa^2}{4(1-\eps t\kappa)^2}\vert v_0\vert^2 ds dt \\&-\frac54\int_{\Str}\frac{(\eps t\kappa')^2}{(1-\eps t\kappa)^4}\vert u\vert^2 dsdt -\frac12\int_{\Str}\frac{\eps t\kappa''}{(1-\eps t\kappa)^3}\vert u\vert^2 ds dt\\= \ &\frac{E_1(m\eps)}{\eps^2}\|v_0\|_{L^2(\Str,\C^2)}^2 - \int_{\Str^{R,\pm}} \frac{\kappa^2}{4(1-\eps t\kappa)^2}\vert v_0\vert^2 ds dt \\&-\frac54\int_{\Str^{R,\pm}}\frac{(\eps t\kappa')^2}{(1-\eps t\kappa)^4}\vert v_0\vert^2 dsdt \\&-\frac12\int_{\Str^{R,\pm}}\frac{\eps t\kappa''}{(1-\eps t\kappa)^3}\vert v_0\vert^2 ds dt,
\end{align*}
where we have taken into account that $v_0$ is supported in $\Str^{R,\pm} := \{(s,t)\in \mathbb{R}^2 : \pm s > R, t \in (-1,1)\}$. This last equality gives
\begin{align*}
	q_m^{R,\pm}(u) \geq& \frac{E_1(m\eps)}{\eps^2}\|u\|_{L^2(\Str,\C^2)}^2-\eta^\pm(R)\|u\|_{L^2(\Str,\C^2)}^2
\end{align*}
with
\[
	\eta^\pm(R) := \sup_{\{\pm s > R\}}\Big\{\frac{\kappa^2(s)}{4(1-\eps \|\kappa\|_{L^\infty(\R)})}+ \frac{5}4\frac{\eps^2\kappa'(s)^2}{(1-\eps \|\kappa\|_{L^\infty(\R)})^4}+ \frac12\frac{\eps |\kappa''(s)|}{(1-\eps \|\kappa\|_{L^\infty(\R)})^3}\Big)\Big\}.
\]
By \ref{itm:A} and by the additional assumptions on $\kappa'$ and $\kappa''$ we get $\eta^\pm(R) \to 0$ when $R \to +\infty$ and the Lemma is proved applying the min-max principle (Proposition \ref{prop:min-max}).
\end{proof}
\subsection{Convergence of min-max values for $m\to+\infty$}
Combining the results of the previous paragraph we can prove the convergence of the min-max values in the large mass limit. This proof relies on the well-established IMS formula.
\begin{proof}[Proof of Theorem \ref{thm:nonrelativistic}] In this proof we assume that $\Gamma$ is of class $C^4$, $\kappa'(s) \to 0$ and $\kappa''(s) \to 0$ when $|s|\to +\infty$.

Consider a partition of unity given by cut-off functions $\theta_1,\theta_2,\theta_3\in C^\infty(\R)$, with $0\leq\theta_k\leq1$, $k=1,2,3$, and such that $\theta_1^2+\theta_2^2+\theta_3^2=1$. We also assume that 
\[
\left\{\begin{aligned}
	\theta_1(s)&=0 \quad \mbox{if $s\geq-\frac{1}{2}$}\,, \\
	\theta_2(s)&=0 \quad \mbox{if $s\leq\frac{1}{2}$}\,, \\
	\theta_3(s)&=0 \quad \mbox{if $\vert s\vert\geq1$}\,.
\end{aligned}\right.
\]
Recall that $U_1$ is the unitary map defined in \eqref{eq:U_1} and for $k\in\{1,2,3\}$, define
\[
\chi_{k,R}:=(U^{-1}_1\theta_{k,R}),
\]
where for $s\in\R$ we have set $\theta_{k,R}(s):=\theta_k(R^{-1}s)$. In particular, arguing as in \eqref{eq:gradchi}, we get for all $x = \gamma(s) + t\eps\nu(s)\in \Omega_\eps$:
\be\label{eq:gradchi_kR}
\nabla\chi_{k,R}(x)=\frac{\theta_{k,R}'(R^{-1}s)}{R(1-\eps t\kappa)}\gamma'(s).
\ee
Let $u=(u_1,u_2)^\top\in\dom{q_m}$, then by Lemma \ref{lem:quadformsquare} and the fact that $\chi_{1,R}^2+ \chi_{2,R}^2 + \chi_{3,R}^2 = 1$ we have
\be\label{eq:splitq_m}
q_m(u)=\sum^3_{k=1}\left(\int_{\Omega_\eps}\vert \chi_{k,R}\nabla u\vert^2\,dx+\int_{\partial\Omega_\eps}(m-\frac{\kappa_\eps}{2})\vert\chi_{k,R}u\vert^2\,ds \right)\,.
\ee
Let us rewrite the first integral in \eqref{eq:splitq_m}. We have
\[
\begin{split}
\int_{\Omega_\eps}\vert\chi_{k,R}\nabla u\vert^2\,dx&=\sum_{j=1}^2\int_{\Omega_\eps}\vert\nabla(\chi_{k,R} u_j)-u_j\nabla\chi_{k,R}\vert^2\,dx\\
&=\sum_{j=1}^2\Bigg\{ \int_{\Omega_\eps}\vert\nabla(\chi_{k,R} u_j)\vert^2\,dx+\int_{\Omega_\eps}| u_j|^2\vert\nabla\chi_{k,R}\vert^2\,dx\\
&\qquad\qquad-2\Re\left(\int_{\Omega_\eps}\langle \nabla(\chi_{k,R}u_j),u_j\nabla\chi_{k,R}\rangle dx\right)\Bigg\} .
\end{split}
\]
Moreover for $j \in\{1,2\}$, there holds
\begin{align*}
2\Re\left(\int_{\Omega_\eps}\langle \nabla(\chi_{k,R}u_j),u_j\nabla\chi_{k,R}\rangle dx\right)=&2\int_{\Omega_\eps}| u_j|^2\vert\nabla\chi_{k,R}\vert^2\,dx\\&\quad+\frac{1}{2}\int_{\Omega_\eps}\langle\nabla(\chi^2_{k,R}),\nabla(| u_j|^2)\rangle dx \,.
\end{align*}
Recall that $\sum^3_{k=1}\chi^2_{k,R}=1$, so that, summing up with respect to $k\in \{1,2,3\}$, the last term in the above formula vanishes. Thus, we find the following IMS formula :
\be\label{eq:q_m} 
q_m(u)=q^{\frac{R}{2},-}_m(\chi_{1,R}u)+q^{\frac{R}{2},+}_m(\chi_{2,R}u)+q^{R}_m(\chi_{3,R}u)-\int_{\Omega_\eps}W_R\vert u\vert^2\,dx\,,
\ee
where $W_R:=\sum^3_{k=1}\vert\nabla\chi_{k,R}\vert^2$ and $\| W_R\|_{L^\infty(\Omega_\eps)}\leq\frac{C}{R^2}$, for some constant $C>0$, by \eqref{eq:gradchi_kR}.

Now, fix $j\in\N$ and consider the isometry
\[
	\mathcal{I} : L^2(\Omega_\eps,\C^2) \to L^2(\Omega_\eps^{\frac{R}2,-},\C^2)\times L^2(\Omega_\eps^{\frac{R}2,-},\C^2)\times L^2(\Omega_\eps^{R},\C^2)
\]
defined by $\mathcal{I} u = (\chi_{1,R} u, \chi_{2,R}u,\chi_{3,R}u)$. Let $W \subset \dom{q_m}$ be a vector space of dimension $j$, by \eqref{eq:q_m}, there holds
\begin{multline*}
	\bigg(\sup_{u \in W \setminus \{0\}}  \frac{q_m(u)}{\|u\|_{L^2(\Omega_\eps,\C^2)}^2}\bigg) + \frac{C}{R^2} \\\geq \sup_{v = (v_1,v_2,v_3)\in (\mathcal{I}W) \setminus\{0\}}\frac{q_m^{\frac{R}2,-}(v_1) + q_m^{\frac{R}2,+}(v_2) + q_m^{R}(v_3)}{\|v_1\|_{L^2(\Omega_\eps^{\frac{R}2,-},\C^2)}^2 + \|v_2\|_{L^2(\Omega_\eps^{\frac{R}2,+},\C^2)}^2 + \|v_3\|_{L^2(\Omega_\eps^{R},\C^2)}^2}.
\end{multline*}
As $\mathcal{I}$ is an isometry we get $\dim(\mathcal{I}W) = j$ and by definition of the cut-off functions $\chi_{k,R}$ ($k\in\{1,2,3\}$), we also have $(\mathcal{I}W) \subset \mathfrak{D} := \dom{q_m^{\frac{R}2,-}}\times\dom{q_m^{\frac{R}2,+}}\times\dom{q_m^{R}}$. In particular, there holds
\begin{multline*}
	\bigg(\sup_{u \in W \setminus \{0\}}  \frac{q_m(u)}{\|u\|_{L^2(\Omega_\eps,\C^2)}^2}\bigg) + \frac{C}{R^2} \\\geq \inf_{\substack{V \subset \mathfrak{D}\\ \dim(V) = j}}\sup_{v = (v_1,v_2,v_3)\in V \setminus\{0\}}\frac{q_m^{\frac{R}2,-}(v_1) + q_m^{\frac{R}2,+}(v_2) + q_m^{R}(v_3)}{\|v_1\|_{L^2(\Omega_\eps^{\frac{R}2,-},\C^2)}^2 + \|v_2\|_{L^2(\Omega_\eps^{\frac{R}2,+},\C^2)}^2 + \|v_3\|_{L^2(\Omega_\eps^{R},\C^2)}^2}.
\end{multline*}
Now, taking the infimum over all vector spaces $W\subset \dom{q_m}$ of dimension $j$ and noting that the right-hand side is the $j$-th min-max value of the quadratic form of the tensor product of the three self-adjoint operators associated with the quadratic forms $q_m^{\frac{R}2,-}$, $q_m^{\frac{R}2,+}$ and $q_m^{R}$ respectively, the min-max principle (Proposition \ref{prop:min-max}) yields:
\begin{align*}
	\mu_j(q_m) + \frac{C}{R^2}\geq \ & \text{ $j$-th smallest element of the set } \\& \ \{\mu_j(q_m^R)\}_{j\in\mathbb{N}} \bigcup \{\mu_j(q_m^{\frac{R}2,+})\}_{j\in \mathbb{N}}\bigcup \{\mu_j(q_m^{\frac{R}2,-})\}_{j\in \mathbb{N}}.
\end{align*}
First, remark that by the min-max principle for all $j\in \N$, $m\mapsto\mu_j(q_m)$ is a non-decreasing function on $[0,+\infty)$ and such that $\mu_j(q_m) \leq \mu_j(q_\infty)$. In particular $\mu_j(q_m)$ has a limit when $m\to +\infty$.

Now, pick $j_0\in\N$ such that $j_0 < N_\Gamma +1$ (with the convention that $N_\Gamma + 1 = +\infty$ if $N_\Gamma = +\infty$). Recall that by Proposition \ref{prop:specdirlapl} $\mu_{j}(q_\infty) < \frac{\pi^2}{4\eps^2}$ for all $j \in \{1,\dots,2j_0\}$. For all $k\in\N$, by Lemma \ref{lem:extlbRgrand}, there holds 
\[
\mu_k(q_m^{\frac{R}{2},\pm})\geq\mu_1(q_m^{\frac{R}{2},\pm})\geq \frac{E_1(m\eps)}{\eps^2}-\eta^\pm(R),
\]
and $\eta^\pm$ does not depend on $m$ and $\eta^\pm(R) \to 0$ when $R\to +\infty$. In particular, if one fixes $\alpha > 0$, there exists $R_0 > 0$ such that for all $R > R_0$ there holds $\eta^\pm(R) < \frac\alpha2$. Now, using \ref{itm:pt5} of Proposition \ref{prop:op1Dess}, there exists $m_0>0$ such that for all $m > m_0$ there holds $\frac{E_1(m\eps)}{\eps^2} \geq \frac{\pi^2}{4\eps^2} - \frac\alpha2$. Choosing $\alpha = \frac{1}4\big(\frac{\pi^2}{4\eps^2} - \mu_{2j_0}(q_\infty)\big)$ it gives
\begin{equation}\label{eqn:borne1}
	\mu_1(q_m^{\frac{R}{2},\pm}) \geq \frac{\pi^2}{4\eps^2} - \frac{1}4\big(\frac{\pi^2}{4\eps^2} - \mu_{2j_0}(q_\infty)\big)
\end{equation}
and by Lemma \ref{lem:dirconv} there exists $m_1 > 0$ such that for all $m \geq m_1$ there holds
\begin{equation}\label{eqn:borne2}
	\mu_j(q_m^R) \leq \mu_j(q_\infty^R) \leq \mu_j(q_\infty) +\frac{1}4\big(\frac{\pi^2}{4\eps^2} - \mu_{2j_0}(q_\infty)\big) \leq \mu_{2j_0}(q_\infty) +\frac{1}4\big(\frac{\pi^2}{4\eps^2} - \mu_{2j_0}(q_\infty)\big).
\end{equation}
As there holds 
\[
	\mu_{2j_0}(q_\infty) +\frac{1}4\big(\frac{\pi^2}{4\eps^2} - \mu_{2j_0}(q_\infty)\big) < \frac{\pi^2}{4\eps^2} - \frac{1}4\big(\frac{\pi^2}{4\eps^2} - \mu_{2j_0}(q_\infty)\big)\, ,
\]
\eqref{eqn:borne1} and \eqref{eqn:borne2} give that for all $m > \max(m_0,m_1)$ and all $R > R_0$ there holds
\[
	\mu_j(q_m) + \frac{C}{R^2}\geq \mu_j(q_m^R).
\]
Hence, taking the limit $m\to +\infty$ then $R\to+\infty$ in the last equation, by Lemma \ref{lem:localconv} and Lemma \ref{lem:dirconv} we obtain
\[
	\lim_{m\to+\infty} \mu_j(q_m) \geq \mu_j(q_\infty).
\]
In particular, if $N_\Gamma = +\infty$, the proof is completed. Now assume that $N_\Gamma < +\infty$ and let $j\geq 2N_\Gamma +1$. Let us prove that $\mu_j(q_m)$ converges to $\frac{\pi^2}{4\eps^2}$. By Proposition~\ref{prop:min-max} and Proposition \ref{prop:specdirlapl}, there holds
\[
	\mu_j(q_m)\leq \mu_j(q_\infty) = \frac{\pi^2}{4\eps^2}.
\]
In particular, let us consider the $j$-th smallest element of the set
\[
	\{\mu_j(q_m^R)\}_{j\in\mathbb{N}} \bigcup \{\mu_j(q_m^{\frac{R}2,+})\}_{j\in \mathbb{N}}\bigcup \{\mu_j(q_m^{\frac{R}2,-})\}_{j\in \mathbb{N}}.
\]
Either there exists $k_0 \geq 2N_\Gamma +1$ such that this element is $\mu_{k_0}(q_m^R)$ or $p_0\in \N$ such that this element is $\mu_{p_0}(q_m^{\frac{R}2,\pm})$. In the first case, there holds:
\[
	-\frac{C}{R^2} \leq \frac{\pi^2}{4\eps^2} - (\mu_j(q_m) + \frac{C}{R^2}) \leq \frac{\pi^2}{4\eps^2} - \mu_{k_0}(q_m^R)\leq \frac{\pi^2}{4\eps^2} - \mu_{2N_\Gamma + 1}(q_m^R).
\]
Now, in the second case, there holds
\begin{align*}
	-\frac{C}{R^2} \leq \frac{\pi^2}{4\eps^2} - (\mu_j(q_m)+\frac{C}{R}^2) \leq \frac{\pi^2}{4\eps^2} - \mu_{p_0}(q_m^{\frac{R}2,\pm})&\leq \frac{\pi^2}{4\eps^2} - \mu_{1}(q_m^{\frac{R}2,\pm})\\&\leq \frac{\pi^2}{4\eps^2} - \frac{E_1(m\eps)}{\eps^2} + \eta^\pm(R),
\end{align*}
where we have used Lemma \ref{lem:extlbRgrand}. These two inequalities yield
\[
	-\frac{C}{R^2} \leq \frac{\pi^2}{4\eps^2} - (\mu_j(q_m)+\frac{C}{R}^2) \leq \min\Big(\frac{\pi^2}{4\eps^2} - \mu_{2N_\Gamma + 1}(q_m^R),\frac{\pi^2}{4\eps^2} - \frac{E_1(m\eps)}{\eps^2} + \eta^\pm(R)\Big)
\]
Now, taking the limit $m\to +\infty$ and then $R\to+\infty$ by Lemma~\ref{lem:localconv}, Lemma~\ref{lem:dirconv}, \ref{itm:pt5} of Proposition \ref{prop:op1Dess} and Lemma \ref{lem:extlbRgrand} we get
\[
\lim_{m\to+\infty}\mu_j(q_m) = \frac{\pi^2}{4\eps^2}
\]
and Theorem \ref{thm:nonrelativistic} is proved.
\end{proof}


\section{A quantitative condition 
for the existence of bound states}\label{sec:quantitative}
The goal of this section is to obtain an explicit geometric condition on the curvature of the base curve $\Gamma$ which ensures that the operator $\cD_\Gamma(\eps,m)$ has at least two bound states.

To state it, whenever $\Gamma$ is of class $C^4$, we introduce 
the well-known geometric potential
(cf.~\cite[Eq.~(3.9)]{ES})
\begin{equation*}
  V_\eps(s,t) :=   
  - \frac{1}{4} \, \frac{\kappa(s)^2}{(1-\eps t\kappa(s))^2}
  - \frac{1}{2} \, \frac{\kappa''(s) \, \eps t}{(1-\eps t\kappa(s))^3}
  - \frac{5}{4} \, \frac{\kappa'(s)^2\,\eps^2 t^2}{(1-\eps t\kappa(s))^4}
  \,.
\end{equation*}
It depends on the geometry of the waveguide~$\Omega_\eps$
through the curvature~$\kappa$ of the base curve~$\Gamma$,
its two derivatives	
and the radius~$\eps$ of the tubular neighbourhood.

The sufficient condition we obtain reads as follows.
\begin{prop}[\emph{Quantitative existence of bound states}]\label{thm:quantitative} 
Let us assume additionally 
that $\Gamma$ is of class~$C^4$ and that
$\supp \kappa \subset (-L,L)$ with $L>0$.
If
\begin{equation}\label{hypothesis} 
  I_\eps := -\int_{\R} \int_{-1}^1  
  V_\eps(s,t) 
  \, \cos^2\left(\frac{\pi}{2}\, t\right) \, d t \, ds
  >  0\,,
\end{equation}
then there exists $m_0 \in \R$ such that for every $m>m_0$,
\begin{equation}\label{essgap}
\Sp_{\mathrm {dis}}(\cD_\Gamma(\varepsilon,m)) \neq \emptyset.
\end{equation}
Moreover, there holds
\begin{equation}\label{eqn:estm_0}
	m_0 \leq \frac1{2\eps}\Big[\frac1{I_\eps^2}\Big(\frac{4\pi^2 L}{3\eps^2} + \frac2L\Big)^2-1
	\Big]
\end{equation}
\end{prop}
Remark that if \eqref{essgap} holds, due to charge conjugation symmetry, we have $\# Sp_{dis}(\cD_\Gamma(\varepsilon,m)) \geq 2$.

Note that the integral~$I_\eps$ is independent of~$m$.
Since $V_\eps(s,t) \to -\frac{1}{4}\kappa(s)^2$ 
as $\eps \to 0$, 
uniformly in $(s,t)\in\R\times(-1,1)$,
the sufficient condition~\eqref{hypothesis} is always satisfied whenever the curvature~$\kappa$ 
is not identically equal to zero and $\eps$~is small enough.

Compared to Proposition \ref{cor:evnumber}, Proposition \ref{thm:quantitative} gives a quantitative geometric bound control on $m_0$ to obtain the existence of bound states.

We work with the square of the operator $\cD_\Gamma(\eps,m)$ studying the min-max value $\mu_1(\cD_\Gamma(\eps,m)^2)$ following the notation introduced in Definition \ref{def:minmax}.
The main idea is that
thanks to Proposition \ref{prop:equiv} and Proposition \ref{prop:quadform2}, we have
\begin{equation}
	\mu_1(\cD_\Gamma(\eps,m)^2) = \inf_{u\in \dom{\cE_\Gamma(\eps,m)}\setminus\{0\}}\frac{\|\cE_\Gamma(\eps,m)u\|_{L^2(\Str,\C^2)}^2}{\|u\|_{L^2(\Str,\C^2)}^2}.
\label{variational}
\end{equation}

\begin{proof}[Proof of Proposition \ref{thm:quantitative}]
In view of~\eqref{variational} and the symmetry of the spectrum of $\cD_\Gamma(\eps,m)$ (see Theorem \ref{prop:sa-essspec}), 
it is enough to find a test function $u \in \dom{\cE_\Gamma(\eps,m)}$
such that
\begin{equation}
  q(u) := \|\cE_\Gamma(\eps,m)u\|_{L^2(\Str,\C^2)}^2 
  - \big(m^2+\eps^{-2}E_1(m\eps)\big) \, \|u\|_{L^2(\Str,\C^2)}^2
  < 0
  \,,
\end{equation}
with $\dom{q} := \dom{\cE_\Gamma(\eps,m)}$. Then, necessarily we have $\mu_1(q) < 0$.%

Fix $\eta > 0$, and define
\begin{equation*}
  u_\eta(s,t) := \frac1{\sqrt{2}}\varphi_\eta(s)\cos\big(\frac\pi2t\big)
  \begin{pmatrix}
   e^{i \frac{\theta(s)}2} \\
   e^{-i \frac{\theta(s)}2}
  \end{pmatrix} 
  \,,
\end{equation*}
where $\theta$ is defined in \eqref{eqn:deftheta} and, for every $\eta \in \R$, 
\begin{equation*}
  \varphi_\eta(s) :=
  \begin{cases}
    1 & \mbox{if} \quad |s| \leq \eta \,, \\
    \displaystyle
    \frac{2\eta-|s|}{\eta} & \mbox{if} \quad \eta < |s| < 2\eta \,, \\
    0 & \mbox{if} \quad |s| \geq 2\eta \,.
  \end{cases}
\end{equation*}
Remark that $u_\eta \in H_0^1(\Str,\C^2) \subset \dom{q}$ and 
$\|u_\eta\|_{L^2(\Str,\C^2)}^2 = \|\varphi_\eta\|_{L^2(\R)}^2 
= \frac{8}{3} \eta$.
Using the boundary condition,
one easily checks the identity
\begin{equation*}
\begin{split}
 \frac1{\eps^2} \int_{\Str} \vert \partial_t u_\eta(s,t)\vert^2 \, ds \, dt  
  + \eps m \int_\R \vert u_\eta(s,-1)\vert^2 \, ds
  + \eps m \int_\R \vert u_\eta(s,1)\vert^2 \, ds\\
  =\frac{\pi^2}{4\eps^2}\|u_\eta\|_{L^2(\Str,\C^2)}^2
  \,.
  \end{split}
\end{equation*}
Consequently, there holds
\begin{align}
 q(u_\eta) =&  \eps^{-2}\Big(\frac{\pi^2}{4} - E_1(m\eps)\Big)\|u_\eta\|_{L^2(\Str,\C^2)}^2 + 
  \int_{\Str} \frac{\vert(\partial_s-i\frac{\kappa}2\sigma_3)u_\eta(s,t)\vert^2}{(1-\eps t\kappa(s))^2} \, ds \, dt\nonumber\\&
  \qquad+ \int_{\Str} V_\eps(s,t) \, \vert u_\eta(s,t)\vert^2 \, ds \, dt
  \,.
  \label{hatform}
\end{align}

To deal with the second term on the right-hand side of~\eqref{hatform}, we set  $v_\eta := e^{-i\frac{\theta}2\sigma_3}u_\eta$ 
and remark that for all $(s,t)\in\Str$ there holds
\[
	v_\eta(s,t) = \frac1{\sqrt{2}}\varphi_\eta(s)\cos\big(\frac\pi2t\big)
  \begin{pmatrix}
   1 \\
   1
   \end{pmatrix},\quad \|v_\eta(s,t)\|_{\C^2} = \|u_\eta(s,t)\|_{\C^2}.
\]
In particular, we remark that
\begin{equation*}
  e^{-i\frac{\theta}2\sigma_3}(\partial_s-i\frac{\kappa}2\sigma_3)u_\eta(s,t) = (\partial_s v_\eta)(s,t)
  = \frac1{\sqrt{2}} \varphi_\eta'(s)\cos\big(\frac\pi2t\big)
  \begin{pmatrix}
   1 \\
   1
   \end{pmatrix}.
\end{equation*}
Consequently, we obtain
\begin{align}\label{tosimplify}
	q(u_\eta) &= \eps^{-2}\Big(\frac{\pi^2}4 - E_1(m\eps)\Big)
	\|\varphi_\eta\|_{L^2(\R)}^2 
	\nonumber \\
	& \qquad
	+ \int_\R  |\varphi_\eta'(s)|^2
	\int_{-1}^1 
  \frac{1}{(1-\kappa(s)\,\eps t)^2} 
  \, \cos^2\left(\frac{\pi}{2}\, t\right) \, d t 
	\, ds
	\\ 
	& \qquad
	+\int_{\R} |\varphi_\eta(s)|^2 \, V_\eps(s,t)
	\, \cos^2\left(\frac{\pi}{2}\, t\right) \, d t 
	\, ds.
	\nonumber
\end{align}
%
%

Now we employ the hypothesis that the curvature~$\kappa$
(and therefore also its derivatives~$\kappa'$ and~$\kappa''$)
is compactly supported and choose $\eta \geq L$.
Then the last line equals~$-I_\eps$
and the second line equals 
$\|\varphi_\eta'\|_{L^2(\R)}^2 = \frac{2}{\eta}$.
In summary,
$$
  q(u_\eta)  
  = \eps^{-2}\Big(\frac{\pi^2}4 - E_1(m\eps)\Big) \frac{8}{3} \eta
  + \frac{2}{\eta}
  - I_\eps
  \,.
$$ 
Using in~\eqref{implicit1} the elementary bound $\tan(x) \leq x - \pi$
valid for every $x \in (\frac\pi2,\pi]$, 
we get the estimate
$$
  \sqrt{E_1(m\eps)} 
  \geq \frac{\pi}{2} \frac{2m\eps}{1+2m\eps}
  \,.
$$
Remark that this lower bound on $E_1(m\eps)$ holds for all masses $m \geq 0$ but there is no reason for it to be optimal for small masses.
Consequently, using the elementary inequality $(1+4m\eps) \leq 2(1 + 2m\eps)$, we get
$$
  q(u_\eta)  
  \leq \frac{\pi^2}{4\eps^2} \frac{1+4m\eps}{(1+2m\eps)^2}
  \frac{8}{3} \eta
  + \frac{2}{\eta}
  - I_\eps \leq \frac{\pi^2}{2\eps^2} \frac{1}{(1+2m\eps)}
  \frac{8}{3} \eta
  + \frac{2}{\eta}
  - I_\eps
 \,.
$$
Setting $\eta := L \sqrt{1+2m\eps}\geq L$, we find
$$
  q(u_\eta) \leq  
  \left(\frac{4\pi^2 L}{3\eps^2} + \frac{2}{L}\right)
  \frac{1}{\sqrt{1+2m\eps}} - I_\eps
  \,.
 $$
Therefore, if $I_\eps > 0$, we see that $q(u_\eta)$ is negative
whenever $m \geq \tilde{m}_0$, where $\tilde{m}_0$ coincides with the right-hand-side of \eqref{eqn:estm_0}. It concludes the proof of Proposition \ref{thm:quantitative}.
\end{proof}
\begin{rem}
The hypothesis that~$\kappa$ is compactly supported 
is apparently just a technical condition 
in order to simplify the expression~\eqref{tosimplify}.
\end{rem}

 \appendix\label{technical}
 
 \section{Proof of some technical results}
In this section we collect the proofs of some technical results stated in the paper, in order to simplify the overall presentation.

\begin{proof}[Proof of Proposition~\ref{prop:op1Dess} and Corollary~\ref{cor:cor1}]
	The multiplication operators by $\sigma_1$ and $\sigma_3$ are bounded and self-adjoint in $L^2\big((-1,1),\C^2\big)$ thus $\cT(k,m)$ is self-adjoint if and only if $\cT_0$ is self-adjoint. An integration by parts easily yields that $\cT_0$ is symmetric and by definition, one has
\begin{multline*}
		\dom{\cT_0^*} = \Big\{ u \in L^2\big((-1,1),\C^2\big) :\exists\ w \in L^2\big((-1,1),\C^2\big) \text{ such that }\\  \forall\ v \in \dom{\cT_0}, \ \langle  u,\cT_0 v\rangle_{L^2((-1,1),\C^2)} = \langle  w,v\rangle_{L^2((-1,1),\C^2)}\Big\}.
\end{multline*}
For every
 $v\in \cD := C_0^\infty\big((-1,1),\C^2\big)$ and $u \in \dom{\cT_0^*}$, there holds
\begin{align*}
	\langle  \cT_0^*u,v\rangle_{L^2((-1,1),\C^2)} = \langle  u,\cT_0 v\rangle_{L^2((-1,1),\C^2)} 
	&= \langle u, -i \sigma_2 v'\rangle_{L^2((-1,1),\C^2)}\\
	& = \langle u, i \overline{\sigma_2v'}\rangle_{\cD',\cD} \\
	&= \langle-i \sigma_2 u',  \overline{v}\rangle_{\cD',\cD}\\
	& = \langle \cT_0^*u,\overline{v}\rangle_{\cD',\cD},
\end{align*}
where $\langle\cdot,\cdot\rangle_{\cD',\cD}$ is the duality bracket of distributions. In particular, we know that $\cT_0^*u = -i \sigma_2 u' \in L^2\big((-1,1),\C^2\big)$ thus we get $u\in H^1\big((-1,1),\C^2\big)$. Moreover, if $v \in \dom{\cT_0}$ there holds
\begin{align*}
	\langle  \cT_0^*u,v\rangle_{L^2((-1,1),\C^2)} &= \langle  -i \sigma_2 u',v\rangle_{L^2((-1,1),\C^2)} \\
	&= \langle   u,-i \sigma_2 v'\rangle_{L^2((-1,1),\C^2)} + \Big[\langle-i \sigma_2 u,v\rangle_{\C^2}\Big]_{-1}^1\\
	& = \langle u,\cT_0 v\rangle_{L^2((-1,1),\C^2)} - u_2(1)\overline{v_1}(1) + u_1(1)\overline{v_2}(1)\\&\qquad + u_2(-1)\overline{v_1}(-1) - u_1(-1)\overline{v_2}(-1).
\end{align*}
Since $v\in\dom{\cT_0}$ we obtain
\[
	0 =  - (u_2(1) + u_1(1))\overline{v_1}(1) + (u_2(-1) -u_1(-1))\overline{v_1}(-1).
\]
This holds for any $v\in\dom{\cT_0}$, so that $u_2(\pm 1) = \mp u_1(\pm 1)$ and $v \in \dom{\cT_0}$. In particular $\cT_0^* = \cT_0$. Observe that, by the closed graph theorem, $\dom{\cT(k,m)}$ is continuously embedded in $H^1\big((-1,1),\C^2\big)$ which itself is compactly embedded in $L^2\big((-1,1),\C^2\big)$. Thus, $\cT(k,m)$ has compact resolvent.

Let us prove Point \ref{itm:pt1} by picking $u\in\dom{\cT(k,m)}$ and considering
\begin{equation}\label{eq:squaredHk}
\begin{split}
\Vert\cT(k,m)u\Vert^2&=\|u'\|_{L^2((-1,1),\C^2)}^2 + (m^2+k^2) \|u\|_{L^2((-1,1),\C^2)}^2 \\&\quad+ 2 mk \Re \big(\langle \sigma_3 u,\sigma_1 u\rangle_{L^2((-1,1),\C^2)}\big) + 2m \Re\big(\langle -i\sigma_2  u',\sigma_3 u\rangle_{L^2((-1,1),\C^2)}\big)\\ & \quad + 2k \Re\big(\langle -i\sigma_2  u',\sigma_1 u\rangle_{L^2((-1,1),\C^2)}\big).
\end{split}
\end{equation}
We rewrite \eqref{eq:squaredHk}, arguing as follows.
Using the anti-commutation rules of Pauli matrices and the boundary condition we get
\[
	 2 \Re \big(\langle \sigma_3 u,\sigma_1 u\rangle_{L^2((-1,1),\C^2)}\big) = 2 \Re\big(\langle -i\sigma_2  u',\sigma_1 u\rangle_{L^2((-1,1),\C^2)}\big) = 0
\]
and
\begin{equation}
	2 \Re\big(\langle -i\sigma_2  u',\sigma_3 u\rangle_{L^2((-1,1),\C^2)}\big) = \|u(1)\|_{\C^2}^2 + \|u(-1)\|_{\C^2}^2.
	\label{eqn:intbordrob}
\end{equation}
In particular, we obtain
\begin{align*}
	\|\cT(k,m) u\|_{L^2((-1,1),\C^2)}^2 &= \|u'\|_{L^2((-1,1),\C^2)}^2 + (m^2+k^2) \|u\|_{L^2((-1,1),\C^2)}^2 \\&\qquad+ m (\|u(1)\|_{\C^2}^2 + \|u(-1)\|_{\C^2}^2)\\&\geq (m^2+k^2) \|u\|_{L^2((-1,1),\C^2)}^2.
\end{align*}
Hence, by the min-max principle (see Proposition \ref{prop:min-max}), if $\lambda \in Sp(\cT(k,m))$, we get $|\lambda|\geq \sqrt{m^2+k^2}$. Moreover, the last inequality is strict. 
Indeed, if~$u$ is an eigenfunction of $\cT(k,m)$ associated with an eigenvalue $\lambda$ such that $|\lambda| = \sqrt{m^2 +k^2}$ we necessarily get that $u$ is a constant $\C^2$-valued function on $(-1,1)$ satisfying the boundary conditions given in \eqref{eqn:defA}. It is a contradiction because it implies that $u = 0$ identically. Hence, $\Sp(\cT(k,m)) \cap [-\sqrt{m^2+k^2},\sqrt{m^2+k^2}] = \emptyset$ and Point \ref{itm:pt1} is proved.

Now, let $\lambda \in \Sp(\cT(k,m))$ and pick an associated eigenfunction $u = (u_1,u_2)^\top \in \dom{\cT(k,m)}$. There holds
\be\label{eq:eigenf}
\left\{\begin{array}{rcl}
	mu_1+ku_2-u'_2 &=& \lambda u_1\,,\\
	ku_1+u'_1-mu_2&= & \lambda u_2\,.
\end{array}\right.
\ee
The second equation gives $(m+\lambda)u_2=ku_1+u'_1$ and multiplying the first line by $(\lambda+m)$ we get
$$
-u''_1=Eu_1\,,\qquad E:=\lambda^2-(m^2+k^2)\,.
$$
Recall that $m\geq0$ and that by Point \ref{itm:pt1} we have $E>0$ for all $k\in\R$. Thus we find
$$
u_1(t)=\alpha\cos\big(\sqrt{E}(t+1)\big)
+\beta\sin\big(\sqrt{E}(t+1)\big),
$$
for some constants $\alpha,\beta \in \C$ and as $m+\lambda \neq 0$ we get
$$
u_2(t)=\frac{1}{\lambda+m}\cos\big(\sqrt{E}(t+1)\big)\big(k\alpha+\sqrt{E}\beta\big)+\frac{1}{\lambda+m}\sin\big(\sqrt{E}(t+1)\big)\big(k\beta-\sqrt{E}\alpha\big).
$$
The boundary condition at $t=-1$ gives
$$
(m+\lambda-k)\alpha-\sqrt{E}\beta = 0\,.
$$
The boundary condition at $t=1$ gives
$$
	\big((m+\lambda+k)\cos(2\sqrt{E}) - \sqrt{E}\sin(2\sqrt{E})\big) \alpha + \big((m+\lambda+k)\sin(2\sqrt{E}) + \sqrt{E}\cos(2\sqrt{E})\big)\beta = 0
	.
$$
To obtain a non-zero eigenfunction $u$, there has to hold
\[
	0 = 	\begin{vmatrix}
			m+\lambda-k & -\sqrt{E}\\
			(m+\lambda+k)\cos(2\sqrt{E}) - \sqrt{E}\sin(2\sqrt{E}) & (m+\lambda+k)\sin(2\sqrt{E}) + \sqrt{E}\cos(2\sqrt{E})
		\end{vmatrix}.
\]
Computing the determinant, we are left with the implicit equation
\begin{equation}
	m \sin(2\sqrt{E}) + \sqrt{E}\cos(2\sqrt{E}) = 0.
	\label{eqn:implicit}
\end{equation}
In particular, it yields that the spectrum of $\cT(k,m)$ is symmetric with respect to the origin and we remark that when $m = k = 0$, we necessarily have $\sqrt{E}  = |\lambda|= p \frac{\pi}4$ (with $p\in\N$) and that in this case, a normalized eigenfunction associated with $\lambda = \pm \frac{\pi}4$ is given by
\[
	u_k^\pm(t) = \frac12\cos\left(k\frac\pi4(t+1)\right)\begin{pmatrix}1\\1\end{pmatrix} \pm \frac12 \sin\left(k\frac\pi4(t+1)\right)\begin{pmatrix}1\\-1\end{pmatrix},
\]
which proves Corollary \ref{cor:cor1}.

Remark that for $m > 0$, a solution $E$ to \eqref{eqn:implicit} verifies $\cos(2\sqrt{E})\neq 0$ and we obtain
\begin{equation}\label{implicit1}
	\tan(2\sqrt{E}) + \frac{\sqrt{E}}m = 0.
\end{equation}	
Now, for $p \in \N_0 = \N \cup\{0\}$ define the line segments $I_0 := [0,\frac{\pi}2)$ and $I_{p+1} = ((2p+1)\frac{\pi}2, (2p+3)\frac{\pi}2)$ 
\begin{equation}\label{eqn:defg}
	g_p : I_p \to \R,\quad g_p(x) = \tan(2x) +\frac{x}m.
\end{equation}
Remark that $g_p'(x) >0$ and in particular, the only solution to $g_0(x) = 0$ is $x=0$. For all $p \geq 1$ we have
\[
	\lim_{x\to (2p-1)\frac{\pi}2^+} g_p(x) = -\infty,\quad g_p(p\pi) = p\frac\pi{m}>0.
\]
In particular, for all $p\geq1$ there is a unique solution $x_p \in I_p$ to $g_p(x) =0$. Moreover, it satisfies $x_p \in\big((2p-1)\frac{\pi}2,p\pi\big)$. Hence, for $p\geq1$ $E_p(m)$ is defined as the unique solution $E$ to $g_p(2\sqrt{E}) = 0$. In particular $E_p(m) \in ((2p-1)^2 \frac{\pi^2}{16},p^2\frac{\pi^2}{4})$ which proves Points \ref{itm:pt2} and \ref{itm:pt3}.

Now, we prove \ref{itm:pt4}. Guided by \eqref{eqn:implicit} we define the $C^\infty$ function
\[
	F :\left\{\begin{array}{lcl}
			\mathbb{R}\times\mathbb{R} &\to& \mathbb{R}\\
			(\mu,m) & \mapsto & 2m\sin(\mu) + \mu\cos(\mu)
		\end{array}\right. .
\]
One remarks that $F(\frac\pi2,0) = 0$ and $\partial_\mu F(\frac\pi2,0) = \frac{\pi}2$. Hence, by the implicit function theorem, there exists $\delta_1,\delta_2 > 0$ and a $C^\infty$ function 
$\mu : (-\delta_1,\delta_1)
\to (\frac\pi2-\delta_2,\frac\pi2+\delta_2)$ 
verifying $\mu(0) = \frac\pi2$ and such that for all $|m|<\delta_1$ there holds $F(\mu(m),m) = 0$. Moreover, when $m\to 0$ there holds
\[
	\mu(m) = \mu(0) + \mu'(0)m + \mathcal{O}(m^2) =  \frac\pi2 + \frac4\pi m + \mathcal{O}(m^2).
\]
Necessarily, for $m > 0$ sufficiently small there holds $E_1(m) = \frac14\mu(m)^2$. Hence, when $m\to 0$, there holds
\[
	E_1(m) = \frac{\pi^2}{16} + m + \mathcal{O}(m^2),
\]
which is precisely Point \ref{itm:pt4}.

Finally, we prove \ref{itm:pt5}. Once again, guided by \eqref{eqn:implicit} we define the $C^\infty$ function
\[
	G :\left\{\begin{array}{lcl}
			\mathbb{R}\times\mathbb{R} &\to& \mathbb{R}\\
			(\mu,\nu) & \mapsto & 2\sin(\mu) + \mu\nu\cos(\mu)
		\end{array}\right. .
\]
One remarks that $G(\pi,0) = 0$ and $\partial_\mu G(\pi,0) = -2$. Hence, by the implicit function theorem, there exists $\delta_1,\delta_2 > 0$ and a $C^\infty$ function 
$\mu : (-\delta_1,\delta_1)
\to (\pi-\delta_2,\pi+\delta_2)$ 
verifying $\mu(0) = \pi$ and such that for all $|\nu|<\delta_1$ there holds $G(\mu(\nu),\nu) = 0$. Moreover, when $\nu\to 0$ there holds
\[
	\mu(\nu) = \mu(0) + \mu'(0)\nu + \mathcal{O}(\nu^2) =  \pi -\frac{\pi}{2} \nu + \mathcal{O}(\nu^2).
\]
Necessarily, for $m > 0$ sufficiently large there holds $E_1(m) = \frac14\mu(m^{-1})^2$. Hence, when $m\to +\infty$, there holds
\[
	E_1(m) = \frac{\pi^2}{4} -\frac{\pi^2}{4m} + O(m^{-2}),
\]
which gives \ref{itm:pt5}.
\end{proof}

\begin{proof}[Proof of Lemma \ref{lem:difficile}] Let $u \in \dom{\cE_0(\eps,m)}$ and remark that there holds 
\begin{align}
	\nonumber\|\cC(\eps,m) u\|_{L^2(\Str,\C^2)}^2 &= \|(-i\sigma_1)\partial_s u + m \sigma_3 u\|_{L^2(\Str,\C^2)}^2\\&\nonumber\qquad + \frac1{\eps^2}\underset{:= A}{\underbrace{\|(-i\sigma_2)\partial_t u - \frac\pi4(P^+ - P^-)u\|_{L^2(\Str,\C^2)}^2}}\\
	&\nonumber\qquad + \frac{1}\eps\underset{:=B}{\underbrace{2\Re(\langle (-i\sigma_1)\partial_s u, (-i\sigma_2)\partial_t u - \frac\pi4(P^+ - P^-)u\rangle_{L^2(\Str,\C^2)})}} \\\nonumber&\qquad+ \frac{m}\eps \underset{:= C}{\underbrace{2\Re(\langle \sigma_3 u, (-i\sigma_2)\partial_t u\rangle_{L^2(\Str,\C^2)})}}\\&\qquad - \frac{m\pi}{4\eps}\underset{:=D}{\underbrace{2\Re(\langle\sigma_3 u, (P^+ - P^-)u\rangle_{L^2(\Str,\C^2)})}}.
	\label{eqn:decpleintermes}
\end{align}
Now, we deal with each term appearing
on the right-hand side of \eqref{eqn:decpleintermes}. For further use, for all $k\geq1$, we set $f_k^\pm := \langle u,u_k^\pm\rangle_{L^2((-1,1),\C^2)}$ and recall that $\Pi_k$ denotes the projector defined in \eqref{eq:a_projectors}. In particular, for all $k\geq1$, there holds
\[
	\|\Pi_k u\|_{L^2(\Str,\C^2)}^2 = \int_{\R}\Big(|f_k^+(s)|^2 + |f_k^-(s)|^2\Big)ds.
\]
\paragraph{\bf {Step 1.}} In this step, we analyze the term $A$ appearing in \eqref{eqn:decpleintermes}.
We remark that
\begin{equation}
	(-i\sigma_2\partial_t -\frac{\pi}4(P^+-P^-))u = \sum_{k\geq 2} \frac{(k-1)\pi}4(f_k^+ u_k^+ - f_k^-u_k^-).
	\label{eqn:devinA}
\end{equation}
In particular, it gives
\begin{equation}
	A = \frac{\pi^2}{16}\sum_{k\geq 2}(k-1)^2 \|\Pi_k u\|_{L^2(\Str,\C^2)}^2.
	\label{eqn:exprA}
\end{equation}
\paragraph{\bf{Step 2.}}
A straightforward computation gives
\[
	-i\sigma_1\partial_s u = \sum_{k\geq 1} -i (f^-_k)' u_k^+ - i (f^+_k)'u_k^-.
\]
In particular, using \eqref{eqn:devinA}, there holds
\begin{multline}
	\langle -i\sigma_1\partial_s u, \big(-i\sigma_2\partial_t  - \frac\pi4(P^+-P^-)\big) u\rangle_{L^2(\Str,\C^2)} \\ =  \frac\pi4\sum_{k\geq 2}(k-1)\big(-i\int_\R(f^-_k)'(s)\overline{f_k^+(s)} ds + i\int_\R (f^+_k)'(s)\overline{f_k^-(s)} ds\big).
	\label{eqn:dvlptB}
\end{multline}
Integrating by parts, we find
\begin{multline*}
	\overline{-i\int_\R(f^-_k)'(s)\overline{f_k^+(s)} ds + i\int_\R (f^+_k)' \overline{f_k^-(s)} ds} \\= i \int_\R \overline{(f^-_k)'(s)}f_k^+(s) ds - i \int_\R \overline{(f^+_k)'(s)} f_k^-(s) ds \\
	=-\left(-i\int_\R (f^-_k)'(s)\overline{f_k^+(s)} ds + i\int_\R (f^+_k)'(s)\overline{f_k^-(s)} ds\right)\,,
\end{multline*}
and then using \eqref{eqn:dvlptB} we get
\begin{multline*}
	\langle -i\sigma_1\partial_s u, \big(-i\sigma_2\partial_t  - \frac\pi4(P^+-P^-)\big) u\rangle_{L^2(\Str,\C^2)} \\= -\langle \big(-i\sigma_2\partial_t  - \frac\pi4(P^+-P^-)\big) u,-i\sigma_1\partial_s u\rangle_{L^2(\Str,\C^2)}.
\end{multline*}
In particular, we obtain
\begin{equation}B = 2\Re(\langle -i\sigma_1\partial_s u, \big(-i\sigma_2\partial_t  - \frac\pi4(P^+-P^-)\big) u\rangle_{L^2(\Str,\C^2)}) = 0.
\label{eqn:Biszero}
\end{equation}
\paragraph{\bf{Step 3.}}
In this step we deal with the term $C$. Integrating by parts as in \eqref{eqn:intbordrob}, we obtain: 
\begin{equation}
C = \int_\R \vert u(s,1)\vert^2 + \vert u(s,-1)\vert^2 ds.
\label{eqn:Cisrobin}
\end{equation}
\paragraph{\bf{Step 4}} It remains to deal with the term $D$. To do so we remark that:
\begin{align}
	\nonumber\langle\sigma_3 u, (P^+-P^-)u\rangle_{L^2(\Str,\C^2)} &= \underset{:= \alpha}{\underbrace{\langle\Pi_1\sigma_3\Pi_1 u, (P^+-P^-)u\rangle_{L^2(\Str,\C^2)}}}\\\nonumber&\qquad + \underset{:= \beta}{\underbrace{\langle\Pi_1^\perp\sigma_3\Pi_1^\perp u, (P^+-P^-)u\rangle_{L^2(\Str,\C^2)}}} \\\nonumber&\qquad +\underset{:=\gamma}{\underbrace{\langle\Pi_1\sigma_3\Pi_1^\perp u, (P^+-P^-)u\rangle_{L^2(\Str,\C^2)}}}\\&\qquad + \underset{:= \delta}{\underbrace{\langle\Pi_1^\perp\sigma_3\Pi_1 u, (P^+-P^-)u\rangle_{L^2(\Str,\C^2)}}}.
	\label{eqn:decCgrec}
\end{align}
Now, in each of the next substep, we deal with the terms appearing on the right-hand side of \eqref{eqn:decCgrec}.
\paragraph{\emph{Substep 4.1}}
Remark that there holds
\begin{align*}
	\alpha &= \langle f_1^+\sigma_3 u_1^+ + f_1^-\sigma_3 u_1^-, f_1^+u_1^+ - f_1^- u_1^-\rangle_{L^2(\Str,\C^2)}\\&  = \langle\sigma_3 u_1^+,u_1^+\rangle_{L^2((-1,1),\C^2)}\|f_1^+\|_{L^2(\R)}^2 - \langle\sigma_3 u_1^-,u_1^-\rangle_{L^2((-1,1),\C^2)}\|f_1^-\|_{L^2(\R)}\\&\qquad - \langle\sigma_3 u_1^+,u_1^-\rangle_{L^2((-1,1),\C^2)}\langle f_1^+,f_1^-\rangle_{L^2(\R)} + \langle\sigma_3 u_1^-,u_1^+\rangle_{L^2((-1,1),\C^2)}\langle f_1^-,f_1^+\rangle_{L^2(\R)}.
\end{align*}
Thanks to \eqref{eqn:effmassobtention} we get
\begin{equation}
	\alpha = \frac2\pi \|\Pi_1 u\|_{L^2(\Str,\C^2)}^2.
	\label{eqn:valalpha}
\end{equation}
\paragraph{\emph{Substep 4.2}} We handle the term $\beta$ by obtaining the following upper-bound thanks to the Cauchy-Schwarz inequality:
\begin{equation}
	|\beta| = |\langle\Pi_1^\perp\sigma_3\Pi_1^\perp u, (P^+-P^-)u\rangle_{L^2(\Str,\C^2)}| \leq \|\Pi_1^\perp u\|_{L^2(\Str,\C^2)}^2.
	\label{eqn:ubbeta}
\end{equation}
\paragraph{\emph{Substep 4.3}}
Now, let us focus on the two off-diagonal terms $\gamma$ and $\delta$. A direct computation shows that 
\[
\langle \sigma_3 u^-_k,u^+_1 \rangle_{\C^2}=-\langle \sigma_3 u^+_k,u^-_1 \rangle_{\C^2}\,,\quad \langle \sigma_3 u^-_k,u^-_1 \rangle_{\C^2}=-\langle \sigma_3 u^+_k,u^+_1 \rangle_{\C^2}\,.
\]
Then we get 
\[
	\Pi_1\sigma_3\Pi_1^\perp u = \big(\sum_{k\geq 2}a_k f_k^+ - b_k f_k^-\big)u_1^+ + \big(\sum_{k\geq 2}b_k f_k^+ - a_k f_k^-\big)u_1^-,
\]
where we have set for $k\geq 2$
\begin{align}
	a_k & := \langle\sigma_3 u_k^+,u_1^+\rangle_{L^2((-1,1),\C^2)} = \frac{4}{\pi}\frac{\sin^2(\frac\pi4(k+1))}{(k+1)},\label{eqn:expra_k}\\ 
	b_k & := \langle\sigma_3 u_k^+,u_1^-\rangle_{L^2((-1,1),\C^2)} = \frac{4}{\pi}\frac{\sin^2(\frac\pi4(k-1))}{(k-1)}.\nonumber
\end{align}
Thus, we find
\begin{equation}
	\gamma = \sum_{k\geq 2} \int_\R\langle(a_k - \sigma_1 b_k)\begin{pmatrix}f_k^+\\f_k^-\end{pmatrix},\begin{pmatrix}f_1^+\\f_1^-\end{pmatrix}\rangle_{\C^2}ds.
	\label{eqn:exprgamma}
\end{equation}
A similar computation gives
\begin{equation}
	\delta = \sum_{k\geq 2} \int_\R\langle\begin{pmatrix}f_1^+\\f_1^-\end{pmatrix},(a_k + \sigma_1 b_k)\begin{pmatrix}f_k^+\\f_k^-\end{pmatrix}\rangle_{\C^2}ds.
	\label{eqn:exprdelta}
\end{equation}
In particular, using \eqref{eqn:exprgamma} and \eqref{eqn:exprdelta} we get
\begin{align}
	\nonumber\gamma + \delta &=2\Re\Big(\sum_{k\geq 2}a_k\int_\R \langle\begin{pmatrix}f_1^+\\f_1^-\end{pmatrix},\begin{pmatrix}f_k^+\\f_k^-\end{pmatrix}\rangle_{\C^2}ds\Big)\\&\qquad + 2i\Im\Big(\sum_{k\geq 2}b_k \int_\R\langle\begin{pmatrix}f_1^+\\f_1^-\end{pmatrix},\sigma_1\begin{pmatrix}f_k^+\\f_k^-\end{pmatrix}\rangle_{\C^2}ds\Big).
	\label{eqn:sumgammadelta}
\end{align}
Using \eqref{eqn:decCgrec}, \eqref{eqn:valalpha} and \eqref{eqn:sumgammadelta} we obtain
\[
	D = \frac{4}\pi \|\Pi_1 u\|_{L^2(\Str,\C^2)}^2 + 2 \Re(\beta) + 4 \Re\Big(\sum_{k\geq 2}a_k\int_\R \langle\begin{pmatrix}f_1^+\\f_1^-\end{pmatrix},\begin{pmatrix}f_k^+\\f_k^-\end{pmatrix}\rangle_{\C^2}ds\Big).
\]
In particular, using the Cauchy-Schwartz inequality we get
\begin{equation}
	D \leq \frac{4}\pi \|\Pi_1 u\|_{L^2(\Str,\C^2)}^2 + 2 |\beta| + 4\sum_{k\geq 2}\Big(|a_k|\|\Pi_1 u\|_{L^2(\Str,\C^2)}\|\Pi_ku\|_{L^2(\Str,\C^2)}\Big).
	\label{eqn:ubD}
\end{equation}
Now, let us fix $c>0$ to be chosen later. For all $a,b\in\mathbb{R}$ and $\eps>0$, we recall the elementary inequality $ab \leq \frac{c\eps}2 a^2 + \frac{1}{2c\eps}b^2$ that we use to get for all $k\geq 2$:
\[
	|a_k|\|\Pi_1 u\|_{L^2(\Str,\C^2)}\|\Pi_ku\|_{L^2(\Str,\C^2)} \leq \frac{c\eps}2a_k^2 \|\Pi_1 u\|_{L^2(\Str,\C^2)}^2 + \frac{1}{2c\eps}\|\Pi_k u\|_{L^2(\Str,\C^2)}^2.
\]
Then, summing up for $k\geq 2$, we get
\begin{align}
	\nonumber\sum_{k\geq 2}\Big(|a_k|\|\Pi_1 u\|_{L^2(\Str,\C^2)}\|\Pi_ku\|_{L^2(\Str,\C^2)}\Big)\leq \ & \frac1{2}c\eps S\|\Pi_1u\|_{L^2(\Str,\C^2)}^2 \\&+ \frac1{2c\eps}\|\Pi_1^\perp u\|_{L^2(\Str,\C^2)}^2\label{eqn:ubsumhorrible},
\end{align}
where we have set $S = \sum_{k\geq 2} a_k^2 < +\infty$ because $a_k^2 = \mathcal{O}(k^{-2})$ when $k\to+\infty$ by \eqref{eqn:expra_k}. Taking into account \eqref{eqn:ubbeta} and \eqref{eqn:ubsumhorrible}, \eqref{eqn:ubD} gives
\begin{equation}
	D \leq (\frac4{\pi}+2cS\eps)\|\Pi_1u\|_{L^2(\Str,\C^2)}^2 + 2(1+\frac1{c\eps}) \|\Pi_1^\perp u\|_{L^2(\Str,\C^2)}^2.
\label{eqn:ubfinalD}
\end{equation}
\paragraph{\bf{Step 5.}} In this step we conclude the proof. Using \eqref{eqn:exprA}, \eqref{eqn:Biszero} and \eqref{eqn:Cisrobin}, \eqref{eqn:decpleintermes} becomes
\begin{align*}
	\|\cC(\eps,m) u\|^2_{L^2(\Str,\C^2)}  = \ & \|(- i\sigma_1 \partial_s + m\sigma_3)u\|_{L^2(\Str,\C^2)}^2 + \frac{\pi^2}{16\eps^2}\sum_{k\geq 2}(k-1)^2\|\Pi_k u\|_{L^2(\Str,\C^2)}^2\\& \qquad+ \frac{m}\eps\int_{\R} \Big(\vert u(s,1)\vert^2 + \vert u(s,-1)\vert^2\Big)ds - \frac{m\pi}{4\eps} D\\= \ & \|(- i\sigma_1 \partial_s + m\sigma_3)u\|_{L^2(\Str,\C^2)}^2 + \frac{\pi^2}{16\eps^2}\sum_{k\geq 2}(k-1)^2\|\Pi_k u\|_{L^2(\Str,\C^2)}^2\\& \qquad+\frac1{\eps^2}\int_\R(\mathfrak{\tau}_{m\eps}(u)(s)-\mathfrak{\tau}_0(u)(s))ds -\frac{m\pi}{4\eps} D\\= \ &\|(- i\sigma_1 \partial_s + m\sigma_3)u\|_{L^2(\Str,\C^2)}^2 + \frac{\pi^2}{16\eps^2}\sum_{k\geq 2}(k-1)^2\|\Pi_k u\|_{L^2(\Str,\C^2)}^2\\& \qquad - \frac{\pi^2}{16\eps^2}\|\Pi_1 u\|^2_{L^2(\Str,\C^2)}+\frac1{\eps^2}\int_{\R}(\mathfrak{\tau}_{m\eps}(u)(s)-\mathfrak{\tau}_0(\Pi_1^\perp u)(s))ds -\frac{m\pi}{4\eps} D,
	\end{align*}
	where the quadratic forms $\mathfrak{\tau}_{\eps m}$ and $\mathfrak{\tau}_0$ are defined in \eqref{eq:1Dquadform}. Notice that in the above formula we used the fact that
	\[
	\mathfrak{\tau}_0(u)(s)=\mathfrak{\tau}_0(\Pi_1 u)(s)+\mathfrak{\tau}_0(\Pi^\perp_1 u)(s)=\frac{\pi^2}{16}\Vert (\Pi_1 u)(s)\Vert^2_{L^2(-1,1,\C^2)}+\mathfrak{\tau}_0(\Pi^\perp_1 u)(s)\,,\qquad s\in\R\,
	\]
	and
	\[
	\|\Pi_1 u\|^2_{L^2(\Str,\C^2)}=\int_{\R}\|(\Pi_1 u)(s)\|^2_{L^2((-1,1),\C^2)}ds\,.
	\]
Using Lemma \eqref{lem:bddbelow1Dform}, this last inequality becomes
\begin{align*}
	\|\cC(\eps,m) u\|^2_{L^2(\Str,\C^2)} \geq \ & \|(- i\sigma_1 \partial_s + m\sigma_3)u\|_{L^2(\Str,\C^2)}^2 + \frac{\pi^2}{16\eps^2}\|\Pi_1^\perp u\|_{L^2(\Str,\C^2)}^2\\& + \frac{1}{\eps^2}\Big(E_1(m\eps) - \frac{\pi^2}{16}\Big)\|\Pi_1 u\|_{L^2(\Str,\C^2)}^2 -\frac{m\pi}{4\eps} D
\end{align*}
and \eqref{eqn:ubfinalD} yields
\begin{align}
\nonumber\|\cC(\eps,m) u\|^2_{L^2(\Str,\C^2)} \geq \ & \|(- i\sigma_1 \partial_s + m\sigma_3)u\|_{L^2(\Str,\C^2)}^2 \\&\nonumber+ \frac1{\eps^2}\Big(\frac{\pi^2}{16} - \frac{m\pi}{2c} - \frac{m\pi}2\eps\Big)\|\Pi_1^\perp u\|_{L^2(\Str,\C^2)}^2\\& + \frac{1}{\eps^2}\Big(E_1(m\eps) - \frac{\pi^2}{16} - m\eps - \frac{m\pi cS}2\eps^2\Big)\|\Pi_1 u\|_{L^2(\Str,\C^2)}^2.
\label{eqn:lbcCfin}
\end{align}
Now, we choose $c > \frac{8m}\pi$ and remark that there exists $\varepsilon_1> 0$ such that for all $\eps \in(0,\eps_1)$ there holds
\begin{equation}
\frac{\pi^2}{16} - \frac{m\pi}{2c} - \frac{m\pi}2\eps > 0.
\label{eqn:lbperppos}
\end{equation}
Moreover, thanks to \ref{itm:pt4} 
of~Proposition \ref{prop:op1Dess}, there exists $\eps_2$ and $K>0$ such that for all $\eps \in (0,\eps_2)$
\begin{equation}
	E_1(m\eps) - \frac{\pi^2}{16} - m\eps - \frac{m\pi cS}2\eps^2 > -K \eps^2.
	\label{eqn:lbeps2}
\end{equation}
Setting $\eps_0 := \min(\eps_1,\eps_2)$ 
and taking into account \eqref{eqn:lbperppos} and \eqref{eqn:lbeps2} in \eqref{eqn:lbcCfin} we obtain that for all $\eps \in (0,\eps_0)$ there holds
\[
	K\|\Pi_1 u\|_{L^2(\Str,\C^2)}^2 + \|\cC(\eps,m)u\|_{L^2(\Str,\C^2)}^2 \geq \|(-i\sigma_1\partial_s + m\sigma_3)u\|_{L^2(\Str,\C^2)}^2.
\]
The proof of Lemma \ref{lem:difficile} is completed remarking that $\|\Pi_1 u\|_{L^2(\Str,\C^2)}^2 \leq \|u\|_{L^2(\Str,\C^2)}^2$.
\end{proof}

\subsection*{Acknowledgment}
The research of D.K.~was partially supported 
by the EXPRO grant No.~20-17749X of
the Czech Science Foundation (GACR).
W.B. is member of {\em Gruppo Nazionale per l'Analisi Ma\-te\-ma\-ti\-ca, la Probabilit\`a e le loro Applicazioni} (GNAMPA) of the {\em Istituto Nazionale di Alta Matematica} (INdAM).



\providecommand{\bysame}{\leavevmode\hbox to3em{\hrulefill}\thinspace}
\providecommand{\MR}{\relax\ifhmode\unskip\space\fi MR }
\providecommand{\MRhref}[2]{%
  \href{http://www.ams.org/mathscinet-getitem?mr=#1}{#2}
}
\providecommand{\href}[2]{#2}

\end{document}